\documentclass[11pt,a4paper]{article}
\usepackage{amssymb}
\usepackage{amsmath}
\usepackage{amsfonts}
\usepackage{bbm}
\usepackage{amsthm}
\usepackage{mathrsfs}
\usepackage{hyperref}
\usepackage{color}
\usepackage[utf8]{inputenc}
\usepackage{stmaryrd}
\usepackage{rotating}
\usepackage{xfrac}
\usepackage{graphicx}
\usepackage{url}

\definecolor{darkred}{rgb}{0.8,0.1,0.1}
\hypersetup{
     colorlinks=false,         
     linkcolor=darkred,
     citecolor=blue,
}

\usepackage{comment}
\usepackage[margin=2.6cm]{geometry}
\usepackage[all,cmtip]{xy}

\usepackage{marvosym}

\theoremstyle{plain}
\newtheorem{theo}{Theorem}[section]
\newtheorem{lem}[theo]{Lemma}
\newtheorem{propo}[theo]{Proposition}
\newtheorem{cor}[theo]{Corollary}

\theoremstyle{definition}
\newtheorem{defi}[theo]{Definition}
\newtheorem{ex}[theo]{Example}

\newtheorem{rem}[theo]{Remark}

\numberwithin{equation}{section}

\def\nn{\nonumber}

\def\Hom{\mathrm{Hom}}

\def\Aut{\mathrm{Aut}}
\def\Mod{\mathrm{Mod}}
\def\Der{\mathrm{Der}}
\def\der{\mathrm{der}}

\def\Set{\mathsf{Set}}

\def\op{\mathrm{op}}
\def\id{\mathrm{id}}
\def\ev{\mathrm{ev}}

\def\MMM{\mathscr{M}}
\def\AAA{\mathscr{A}}
\def\SSS{\mathscr{S}}
\def\GGG{\mathscr{G}}
\def\OO{\mathcal{O}}

\def\bbZ{\mathbb{Z}}

\def\bbK{\mathbb{K}}
\def\bbC{\mathbb{C}}
\def\bbT{\mathbb{T}}

\def\1{\mathbbm{1}}

\def\sk{\vspace{2mm}}

\newcommand{\bol}[1]{{\boldsymbol{#1}}}

\title{%
Mapping spaces and automorphism groups\\
of toric noncommutative spaces
}

\author{%
Gwendolyn E. Barnes$^{1,a}$, Alexander Schenkel$^{2,3,b}$ \ and \ Richard J.\ Szabo$^{1,c}$\vspace{4mm}\\
{\small ${}^1$ Department of Mathematics, Heriot-Watt University, Edinburgh EH14 4AS, United Kingdom.}\vspace{0.5mm}\\
{\small \& Maxwell Institute for Mathematical Sciences, Edinburgh, United Kingdom.}\vspace{0.5mm}\\
{\small \& The Higgs Centre for Theoretical Physics, Edinburgh, United Kingdom.}\vspace{2mm}\\
{\small ${}^2$ Fakult\"at f\"ur Mathematik, Universit\"at Regensburg, 93040 Regensburg, Germany.}\vspace{2mm}\\
{\small ${}^3$ School of Mathematical Sciences, University of Nottingham,}\\
{\small University Park, Nottingham NG7 2RD, United Kingdom.}\vspace{4mm}\\
 {\small \texttt{email:} $^a$ \texttt{geb31@hw.ac.uk} ~,~$^b$  \texttt{alexander.schenkel@nottingham.ac.uk} ~,~$^c$ \texttt{R.J.Szabo@hw.ac.uk} }
 }

\date{March 2017}

\begin{document}

\maketitle

\begin{abstract}
\noindent
We develop a sheaf theory approach to toric noncommutative geometry which allows us to formalize the concept of mapping spaces between two toric noncommutative spaces. As an application we study the `internalized' automorphism group of a toric noncommutative space and show that its Lie algebra has an elementary description in terms of braided derivations.
\end{abstract}

\paragraph*{Report no.:} EMPG--16--14


\setcounter{tocdepth}{1}
\tableofcontents

\section{\label{sec:intro}Introduction and summary}
Toric noncommutative spaces are among the most studied and best understood examples in
noncommutative geometry. Their function algebras $A$ carry a coaction of a torus Hopf algebra $H$,
whose cotriangular structure dictates the commutation relations in $A$. Famous examples
are given by the noncommutative tori~\cite{Rieffel}, the Connes-Landi spheres \cite{CL} and related
deformed spaces \cite{CDV}. More broadly, toric noncommutative spaces can be regarded as special 
examples of noncommutative spaces that are obtained by Drinfeld twist (or $2$-cocycle) deformations
of algebras carrying a Hopf algebra (co)action, see e.g.\ \cite{AS,BSS1,BSS2} and references therein.
For an algebraic geometry perspective on toric noncommutative varieties, see \cite{CLS1}.
\sk

Noncommutative differential geometry on toric noncommutative spaces, 
and more generally on noncommutative spaces obtained by Drinfeld twist 
deformations, is far developed and well understood. 
Vector bundles (i.e.\ bimodules over $A$) have been studied in \cite{AS}, where also a theory
of noncommutative connections on bimodules was developed. These results were 
later formalized within the powerful framework of closed braided monoidal categories
and thereby generalized to certain nonassociative spaces (obtained by cochain twist deformations)
in~\cite{BSS1,BSS2}. Examples of noncommutative principal bundles (i.e.\ Hopf-Galois extensions \cite{BM,BJM})
in this framework were studied in \cite{LvS}, and these constructions
were subsequently abstracted and generalized in
\cite{ABPS}. In applications to noncommutative gauge theory, 
moduli spaces of instantons on toric noncommutative spaces
were analyzed in~\cite{BL,CLS2,BLvS,CLS14}, while analogous moduli spaces of
self-dual strings in higher noncommutative gauge theory were considered by~\cite{MS}.
\sk

Despite all this recent progress in understanding the geometry of toric noncommutative spaces,
there is one very essential concept missing: Given two toric noncommutative spaces, say
$X$ and $Y$, we would like to have a `space of maps' $Y^X$ from $X$ to $Y$.
The problem with such mapping spaces is that they will in general be
`infinite-dimensional', just like the space of maps between two finite-dimensional manifolds
is generically an infinite-dimensional manifold. In this paper we propose a framework
where such `infinite-dimensional' toric noncommutative spaces may be formalized
and which in particular allows us to describe the space of maps between any two toric 
noncommutative spaces. Our approach makes use of sheaf theory: Denoting by ${}^H\SSS$
the category of toric noncommutative spaces, we show that there is a natural site structure
on ${}^H\SSS$ which generalizes the well-known Zariski site of algebraic geometry to the toric noncommutative setting.
The category of generalized toric noncommutative spaces is then given by the sheaf topos 
${}^H\GGG:= \mathrm{Sh}({}^H\SSS)$  and we show that there is
a fully faithful embedding ${}^H\SSS\to {}^H\GGG$ which allows 
us to equivalently regard toric noncommutative spaces as living in this bigger category.
The advantage of the bigger category ${}^H\GGG$ is that it enjoys very good categorical properties,
in particular it admits all exponential objects. We can thereby make sense  of the `space of maps'
$Y^X$ as a generalized toric noncommutative space in ${}^H\GGG$, i.e.\ as a sheaf on the site ${}^H\SSS$. As an application,
we study the `internalized' automorphism group $\Aut(X)$ of a toric noncommutative space $X$,
which is a certain subobject in ${}^H\GGG$ of the self-mapping space
$X^X$. Using synthetic geometry techniques,
we are able to compute the Lie algebra of $\Aut(X)$ and we show that it can be identified with
the braided derivations considered in \cite{AS,BSS2}. Hence our concept of automorphism groups
`integrates' braided derivations to finite (internalized) automorphisms, which is an open problem in
toric noncommutative geometry that cannot be solved by more elementary techniques. 
\sk

Besides giving rise to a very rich concept of `internalized' automorphism groups
of toric noncommutative spaces, there are many other applications and problems 
which can be addressed with our sheaf theory approach to toric noncommutative geometry.
For example, the mapping spaces $Y^X$ may be used to describe the spaces of field configurations
for noncommutative sigma-models, see e.g.\ \cite{DKL1,DKL2,MR,DLL}. Due to the fact that
the mapping space $Y^X$ captures many more maps than the {\em set} of morphisms $\Hom(X,Y)$ 
(compare with Example \ref{ex:mappingline} in the main text), this will lead to a much richer structure
of noncommutative sigma-models than those discussed previously. Another immediate application
is to noncommutative principal bundles:
It was observed in~\cite{BM} that the definition
of a good notion of gauge transformations for noncommutative Hopf-Galois extensions is somewhat problematic,
because there are in general not enough algebra automorphisms of the total space algebra. 
To the best of our knowledge, this problem has not yet been solved. 
Using our novel sheaf theory techniques, we can give a natural definition
of an `internalized' gauge group for toric noncommutative principal bundles
$P\to X$ by carving out a subobject in ${}^H\GGG$ 
of the `internalized' automorphism group $\Aut(P)$ of the total space
which consists of all maps that preserve the 
structure group action and the base space. 
\sk

The outline of the remainder of this paper is as follows:
In Section \ref{sec:prelim} we recall some preliminary results concerning
cotriangular torus Hopf algebras $H$ and their comodules,
which form symmetric monoidal categories ${}^H\MMM$.
In Section \ref{sec:algebras} we study algebra objects in ${}^H\MMM$
whose commutation relations are controlled by the cotriangular structure on $H$.
We establish a category of finitely-presented algebra objects ${}^H\AAA_{\mathrm{fp}}$,
which contains noncommutative tori, Connes-Landi spheres and related examples,
and study its categorical properties, including coproducts, pushouts and localizations.
The category of toric noncommutative spaces ${}^H\SSS$ is then given by the opposite 
category of ${}^H\AAA_{\mathrm{fp}}$ and we show in Section \ref{sec:spaces}
that ${}^H\SSS$ can be equipped with the structure of a site.
In Section \ref{sec:genspaces} we introduce and study 
the sheaf topos ${}^H\GGG$ whose objects are sheaves on ${}^H\SSS$ which
we interpret as generalized toric noncommutative spaces. We show that the Yoneda embedding factorizes
through ${}^H\GGG$ (i.e.\ that our site is subcanonical) and hence obtain a fully faithful embedding
${}^H\SSS\to {}^H\GGG$ of toric noncommutative spaces into generalized toric noncommutative spaces.
An explicit description of the exponential objects $Y^X$ in ${}^H\GGG$ is given,
which in particular allows us to formalize and study the mapping space 
between two toric noncommutative spaces. Using a simple example, 
it is shown in which sense the mapping spaces $Y^X$ are richer than 
the morphism sets $\Hom(X,Y)$ (cf.\ Example \ref{ex:mappingline}).
In Section \ref{sec:automorphism} we apply these techniques to define
an `internalized' automorphism group $\Aut(X)$ of a toric noncommutative space $X$,
which arises as a certain subobject in ${}^H\GGG$ of the self-mapping space $X^X$.
It is important to stress that $\Aut(X)$ is in general not representable, i.e.
it has no elementary description in terms of a Hopf algebra and hence it is a truly 
generalized toric noncommutative space described by a sheaf on ${}^H\SSS$.
The Lie algebra of $\Aut(X)$ is computed in Section \ref{sec:liealgebras}
by using techniques from synthetic (differential) geometry \cite{MoerdijkReyes,Lavendhomme,Kock}.
We then show in Section \ref{sec:comparison} that the Lie algebra of $\Aut(X)$
can be identified with the braided derivations of the function algebra of $X$. Hence,
in contrast to $\Aut(X)$, its Lie algebra of infinitesimal automorphisms has an elementary description. This identification
is rather technical and it relies on a fully faithful embedding 
${}^H\MMM_{\mathrm{dec}}\to \mathrm{Mod}_{\underline{K}}({}^H\GGG)$ 
of a certain full subcategory (called decomposables) of the category of left $H$-comodules ${}^H\MMM$
into the category of $\underline{K}$-module objects in the sheaf topos ${}^H\GGG$, where
$\underline{K}$ denotes the line object in this topos; the technical
details are presented in Appendix \ref{app:technical}.

\section{\label{sec:prelim}Hopf algebra preliminaries}
In this paper all vector spaces will be over a fixed field $\bbK$
and the tensor product of vector spaces will be denoted simply by $\otimes$.
\sk

The Hopf algebra $H:= \OO(\bbT^n)$ of functions on the algebraic $n$-torus $\bbT^n$
is defined as follows: As a vector space, $H$ is spanned by the basis
\begin{flalign}
\big\{t_{\bol{m}} \,  :\,  \bol{m} = (m_1,\dots,m_n)\in\bbZ^{n}\big\}~,
\end{flalign}
on which we define a (commutative and associative) product and unit by
\begin{flalign}
 t_{\bol{m}}\,t_{\bol{m^\prime}} = t_{\bol{m} + \bol{m^\prime}}\quad,\qquad 
 \1_H^{} = t_{\bol{0}}~.
\end{flalign}
The (cocommutative and coassociative) coproduct, counit and antipode in $H$ are given by
\begin{flalign}
\Delta(t_{\bol{m}}) = t_{\bol{m}}\otimes t_{\bol{m}}\quad,\qquad \epsilon(t_{\bol{m}}) =1
\quad,\qquad S(t_{\bol{m}}) = t_{-\bol{m}}~.
\end{flalign}
We choose a cotriangular structure on $H$, i.e.\ 
a linear map $R : H\otimes H\to \bbK$ satisfying
\begin{subequations}\label{eqn:Rmatrixproperties}
\begin{flalign}
R(f\,g \otimes h) &= R(f\otimes h_{(1)}) ~R(g\otimes h_{(2)}) ~,\\[4pt]
R(f\otimes g\,h) &= R(f_{(1)} \otimes h)~R(f_{(2)}\otimes g)~,\\[4pt]
\epsilon(h)\,\epsilon(g)&=R(h_{(1)}\otimes g_{(1)}) ~R(g_{(2)}\otimes h_{(2)}) ~,
\end{flalign}
\end{subequations}
for all $f,g,h\in H$, where we have used Sweedler notation $\Delta(h) = h_{(1)}\otimes h_{(2)}$ 
(with summation understood) for the coproduct in $H$.
The quasi-commutativity condition $g_{(1)}\,h_{(1)} \,R(h_{(2)}\otimes g_{(2)}) =
R(h_{(1)}\otimes g_{(1)})\, h_{(2)}\, g_{(2)}$, for all $g,h\in H$, 
is automatically fulfilled because $H$ is commutative and cocommutative. 
For example, if $\bbK=\bbC$ is the field of complex numbers, 
we may take the usual cotriangular structure defined by
\begin{flalign}\label{eqn:Rmatrixexplicit}
R(t_{\bol{m}}\otimes t_{\bol{m^\prime}}) = 
\exp\Big({\rm i}\,\sum_{j,k=1}^n m_j \,\Theta^{jk}\, m^\prime_{k}\Big)~,
\end{flalign}
where $\Theta$ is an antisymmetric real $n{\times} n$-matrix,
which plays the role of deformation parameters for the theory.
\sk

Let us denote by ${}^H\MMM$ the category of left $H$-comodules.
An object in ${}^H\MMM$ is a pair $(V,\rho^V)$, where
$V$ is a vector space and $\rho^{V} : V\to H\otimes V$
is a left $H$-coaction on $V$, i.e.\ a linear map satisfying
\begin{flalign}\label{eqn:coactionproperties}
(\id_H\otimes \rho^V)\circ \rho^V = (\Delta\otimes\id_V)\circ \rho^V\quad ,\qquad (\epsilon\otimes \id_V)\circ \rho^V =\id_V~~.
\end{flalign}
We follow the usual abuse of notation and denote objects $(V,\rho^V)$
in ${}^H\MMM$ simply by $V$ without displaying the coaction explicitly. 
We further use a Sweedler-like notation $\rho^V(v) = v_{(-1)}\otimes
v_{(0)}$ (with summation understood)
for the left $H$-coactions. Then \eqref{eqn:coactionproperties} reads as
\begin{flalign}\label{eqn:coactionpropertiesSweedler}
v_{(-1)} \otimes {v_{(0)}}_{(-1)}\otimes {v_{(0)}}_{(0)} = {v_{(-1)}}_{(1)} \otimes {v_{(-1)}}_{(2)}\otimes v_{(0)}
\quad,\qquad \epsilon(v_{(-1)}) \, v_{(0)} = v~.
\end{flalign}
A morphism $L : V\to W$ in ${}^H\MMM$
is a linear map preserving the left $H$-coactions, i.e.\
\begin{subequations}\label{eqn:equivariantmap}
\begin{flalign}
(\id_H\otimes L)\circ\rho^{V} = \rho^W\circ L~,
\end{flalign}
or in the Sweedler-like notation
\begin{flalign}
v_{(-1)}\otimes L(v_{(0)}) =L(v)_{(-1)} \otimes  L(v)_{(0)}~,
\end{flalign}
\end{subequations}
for all $v\in V$.
\sk

The category ${}^H\MMM$ is a monoidal category
with tensor product of two objects $V$ and $W$ 
given by the tensor product $V\otimes W$ of vector spaces
equipped with the left $H$-coaction
\begin{flalign}\label{eqn:tensorcoaction}
\rho^{V\otimes W} : V\otimes W \longrightarrow H\otimes V\otimes W~,~~v\otimes w \longmapsto 
v_{(-1)}\,w_{(-1)}\otimes v_{(0)}\otimes w_{(0)}~.
\end{flalign}
The monoidal unit in ${}^H\MMM$ is given by the one-dimensional vector space
$\bbK$ with trivial left $H$-coaction $\bbK\to H\otimes \bbK\,,~c\mapsto \1_H^{}\otimes c$.
The monoidal category ${}^H\MMM$ is symmetric with commutativity constraint
\begin{flalign}\label{eqn:tauflip}
\tau_{V,W}^{} : V\otimes W \longrightarrow W\otimes V~,~~
v\otimes w \longmapsto R(w_{(-1)}\otimes v_{(-1)})\, w_{(0)}\otimes v_{(0)}~,
\end{flalign}
for any two objects $V$ and $W$ in ${}^H\MMM$.

\section{\label{sec:algebras}Algebra objects}
We are interested in spaces whose algebras of functions
are described by certain algebra objects in the symmetric monoidal category ${}^H\MMM$.
An algebra object in ${}^H\MMM$ is an object $A$ in ${}^H\MMM$
together with two ${}^H\MMM$-morphisms
$\mu_{A}^{} : A\otimes A\to A$ (product) and $\eta_{A}^{}: \bbK \to A$ (unit)
such that the diagrams
\begin{flalign}
\xymatrix@C=3.5em{
\ar[d]_-{\id_{A}\otimes\mu_{A}^{}}A\otimes A\otimes A\ar[r]^-{\mu_{A}^{}\otimes \id_A} & A\otimes A\ar[d]^-{\mu_{A}^{}} & \bbK\otimes A\ar[d]_-{\eta_A^{} \otimes \id_A}\ar[dr]^-{\simeq} & &\ar[dl]_-{\simeq} A\otimes \bbK\ar[d]^-{\id_A\otimes\eta_A^{}}\\
A\otimes A\ar[r]_-{\mu_{A}}& A & A\otimes A \ar[r]_-{\mu_A^{}} &  A & \ar[l]^-{\mu_A^{}}A\otimes A
}
\end{flalign}
in ${}^H\MMM$ commute.
Because ${}^H\MMM$ is symmetric, we may additionally demand that the product $\mu_A^{}$
is compatible with the commutativity constraints in ${}^H\MMM$, i.e.\ the diagram
\begin{flalign}\label{eqn:comalg}
\xymatrix@C=1.5em{
\ar[rd]_-{\mu_A^{}}A\otimes A\ar[rr]^-{\tau_{A,A}^{}} && A\otimes A\ar[dl]^-{\mu_A^{}}\\
&A&
}
\end{flalign}
in ${}^H\MMM$ commutes. This amounts to demanding the commutation relations
\begin{flalign}\label{eqn:commutationrelations}
a\,a^\prime = R(a^\prime_{(-1)}\otimes a_{(-1)})~a^\prime_{(0)}\,a_{(0)}~,
\end{flalign}
for all $a,a^\prime\in A$, where we have abbreviated the product by
$\mu_A^{}(a\otimes a^\prime\, ) = a\,a^\prime$; in the following we
shall also use the compact notation $\1 _A^{}:= \eta_A^{}(1)\in A$ for
the unit element in $A$, or sometimes just $\1$.
Such algebras are {\em not} commutative in the ordinary sense
once we choose a non-trivial cotriangular structure as for example in \eqref{eqn:Rmatrixexplicit}, 
see also Example \ref{ex:algebras}.
\sk

Let us introduce the category of algebras of interest.
\begin{defi}
The category ${}^H\AAA$ has as objects all algebra objects in ${}^H\MMM$ which satisfy the commutativity constraint
\eqref{eqn:comalg}. The morphisms between two objects are all ${}^H\MMM$-morphisms
$\kappa : A\to B$ which preserve products and units, i.e.\ for which
$\mu_B^{}\circ \kappa\otimes\kappa = \kappa\circ \mu_A^{}$ and
$\kappa\circ \eta_A^{} = \eta_B^{}$.
\end{defi}

There is the forgetful functor
$\mathrm{Forget} : {}^H\AAA \to {}^H\MMM$ 
which assigns to any object in ${}^H\AAA$ its underlying left $H$-comodule,
i.e.\ $(A,\mu_A^{} ,\eta_A^{}) \mapsto A$. This functor has a left adjoint
$\mathrm{Free} : {}^H\MMM \to {}^H\AAA$ which describes the free ${}^H\AAA$-algebra
construction: Given any object $V$ in ${}^H\MMM$
we consider the vector space
\begin{flalign}
\mathcal{T}V := \bigoplus_{n\geq 0} \, V^{\otimes n}~,
\end{flalign}
with the convention $V^{\otimes 0} :=\bbK$. Then $\mathcal{T}V$ is a left $H$-comodule
when equipped with the coaction $\rho^{\mathcal{T}V} : \mathcal{T}V \to H\otimes \mathcal{T}V$
specified by
\begin{flalign}
\rho^{\mathcal{T}V}\big(v_1\otimes \cdots\otimes v_n\big) =
{v_{1}}_{(-1)} \cdots {v_{n}}_{(-1)} \otimes {v_{1}}_{(0)}\otimes\cdots\otimes {v_{n}}_{(0)}~.
\end{flalign}
Moreover, $\mathcal{T}V$ is an algebra object in ${}^H\MMM$ when equipped with
the product $\mu_{\mathcal{T}V}^{} : \mathcal{T}V\otimes\mathcal{T}V \to \mathcal{T}V$
specified by
\begin{flalign}
\mu_{\mathcal{T}V}^{}\big((v_1\otimes\cdots \otimes v_n)\otimes (v_{n+1}\otimes \cdots\otimes v_{n+m})\big)=
v_1\otimes \cdots\otimes v_{n+m}
\end{flalign}
and the unit $\eta_{\mathcal{T}V}^{} :\bbK \to \mathcal{T}V$ given by
\begin{flalign}
\eta_{\mathcal{T}V}^{}(c)= c \in V^{\otimes 0} \subseteq \mathcal{T}V~.
\end{flalign}
The algebra object $\mathcal{T}V$ does not satisfy the commutativity constraint \eqref{eqn:comalg},
hence it is not an object of the category ${}^H\AAA$. We may enforce the commutativity constraint
by taking the quotient of $\mathcal{T}V$ by the two-sided ideal $I\subseteq \mathcal{T}V$ generated by
\begin{flalign}
v \otimes v^\prime - R(v^\prime_{(-1)}\otimes v_{(-1)}) ~v^\prime_{(0)}\otimes v_{(0)}~,
\end{flalign}
for all $v,v^\prime\in V$. The ideal $I$ is stable under the left $H$-coaction, i.e.\ 
$\rho^{\mathcal{T}V} : I \to H\otimes I$. Hence the quotient
\begin{flalign}
\mathrm{Free}(V) := \mathcal{T}V /I~
\end{flalign}
is an object in ${}^H\AAA$ when equipped with the induced left $H$-coaction, product and unit.
Given now any ${}^H\MMM$-morphism $L :V\to W$, we define
an ${}^H\AAA$-morphism $\mathrm{Free}(L) : \mathrm{Free}(V)\to \mathrm{Free}(W)$
by setting
\begin{flalign}
\mathrm{Free}(L) \big(v_1\otimes\cdots \otimes v_n\big) = L(v_1)\otimes\cdots \otimes L(v_n)~.
\end{flalign}
This is compatible with the quotients because of \eqref{eqn:equivariantmap}. 
Finally, let us confirm that $\mathrm{Free}: {}^H\MMM\to {}^H\AAA$
is the left adjoint of the forgetful functor $\mathrm{Forget} : {}^H\AAA\to {}^H\MMM$,
i.e.\ that there exists a (natural) bijection
\begin{flalign}
\Hom_{{}^H\AAA}^{}\big(\mathrm{Free}(V), A \big) \simeq \Hom_{{}^H\MMM}^{}\big(V,\mathrm{Forget}(A)\big)
\end{flalign}
between the morphism sets, for any object $V$ in ${}^H\MMM$ and any object $A$ in ${}^H\AAA$.
This is easy to see from the fact that any ${}^H\AAA$-morphism $\kappa : \mathrm{Free}(V)\to A$
is uniquely specified by its restriction to the vector space $V = V^{\otimes 1}\subseteq \mathrm{Free}(V)$ of generators
and hence by an ${}^H\MMM$-morphism $V\to \mathrm{Forget}(A)$.
\sk

From a geometric perspective, the free ${}^H\AAA$-algebras 
$\mathrm{Free}(V)$ describe the function algebras on toric noncommutative planes.
In order to capture a larger class of toric noncommutative spaces,
we introduce a suitable concept of ideals for ${}^H\AAA$-algebras.
\begin{defi}
Let $A$ be an object in ${}^H\AAA$. An ${}^H\AAA$-ideal $I$ of $A$ 
is a two-sided ideal $I\subseteq A$ of the algebra underlying $A$
which is stable under the left $H$-coaction, i.e.\ the coaction $\rho^A$
induces a linear map $\rho^A : I \to H\otimes I$.
\end{defi}

This definition immediately implies 
\begin{lem}
If $A$ is an object in ${}^H\AAA$ and $I$ is an ${}^H\AAA$-ideal of $A$, 
the quotient $A/I$ is an object in ${}^H\AAA$ when equipped with the induced coaction, product and unit.
\end{lem}

This lemma allows us to construct a variety of ${}^H\AAA$-algebras 
by taking quotients of free ${}^H\AAA$-algebras by suitable ${}^H\AAA$-ideals.
We are particularly interested in the case where the object $V$ in ${}^H\MMM$
that underlies the free ${}^H\AAA$-algebra $\mathrm{Free}(V)$ is
finite-dimensional; geometrically, this corresponds to a finite-dimensional toric noncommutative plane.
We  shall introduce a convenient notation for this case: First, notice that the one-dimensional
left $H$-comodules over the torus Hopf algebra $H = \OO(\bbT^n)$
can be characterized by a label $\bol{m}\in\bbZ^n$. The corresponding left $H$-coactions are given by
\begin{flalign}\label{eqn:1DHMMM}
\rho^{\bol{m}} : \bbK\longrightarrow H\otimes \bbK~,~~c\longmapsto t_{\bol{m}}\otimes c~.
\end{flalign}
We shall use the notation $\bbK_{\bol{m}} := (\bbK,\rho^{\bol{m}})$
for these objects in ${}^H\MMM$. The coproduct $\bbK_{\bol{m}}\sqcup \bbK_{\bol{m^\prime}}$ 
of two such objects is given by the vector space $\bbK\oplus \bbK\simeq \bbK^2$
together with the component-wise coaction, i.e.\
\begin{flalign}
\rho^{\bbK_{\bol{m}}\sqcup \bbK_{\bol{m^\prime}}} \big(c\oplus 0\big) = t_{\bol{m}}\otimes (c\oplus 0)\quad,\qquad
\rho^{\bbK_{\bol{m}}\sqcup \bbK_{\bol{m^\prime}}} \big(0\oplus c\big) = t_{\bol{m^\prime}}\otimes (0\oplus c)~.
\end{flalign}
The free ${}^H\AAA$-algebra corresponding to a finite coproduct of objects $\bbK_{\bol{m}_i}$, for $i=1,\dots,N$,
in ${}^H\MMM$ will be used frequently in this paper. Hence we introduce
the compact notation
\begin{flalign}
F_{\bol{m}_1,\dots, \bol{m}_N} := \mathrm{Free}\big(\bbK_{\bol{m}_1}\sqcup \cdots\sqcup \bbK_{\bol{m}_N}\big)~.
\end{flalign}
By construction, the ${}^H\AAA$-algebras $F_{\bol{m}_1,\dots, \bol{m}_N}$ are generated
by $N$ elements $x_i\in F_{\bol{m}_1,\dots,\bol{m}_N}$
whose transformation property under the left $H$-coaction is given by
$\rho^{F_{\bol{m}_1,\dots, \bol{m}_N}}(x_i) = t_{\bol{m}_i}\otimes x_i$
and whose commutation relations read as
\begin{flalign}
x_i \,x_j = R(t_{\bol{m}_j}\otimes t_{\bol{m}_i})~x_j\,x_i~,
\end{flalign}
for all $i,j=1,\dots,N$.
\sk

We can now introduce the category of finitely presented ${}^H\AAA$-algebras.
\begin{defi}
An object $A$ in ${}^H\AAA$ is finitely presented if it is isomorphic
to the quotient $F_{\bol{m}_1,\dots, \bol{m}_N}/I$ of a free ${}^H\AAA$-algebra 
$F_{\bol{m}_1,\dots, \bol{m}_N}$ by an ${}^H\AAA$-ideal $I= (f_k)$ 
that is generated by a finite number of elements $f_k\in F_{\bol{m}_1,\dots, \bol{m}_N}$,
for $k=1,\dots, M$, with $\rho^{F_{\bol{m}_1,\dots, \bol{m}_N}}(f_k) = t_{\bol{n}_k}\otimes f_k$, 
for some $\bol{n}_k\in\bbZ^n$.
We denote by ${}^H\AAA_{\mathrm{fp}}$ the full subcategory of ${}^H\AAA$ whose objects are
all finitely presented ${}^H\AAA$-algebras.
\end{defi}
\begin{ex}\label{ex:algebras}
Let us consider the case $\bbK=\bbC$ and $R$ given by \eqref{eqn:Rmatrixexplicit}.
Take the free ${}^H\AAA$-algebra generated by $x_i$ and $x_i^\ast$, for $i=1,\dots,N$,
with left $H$-coaction specified by $x_i\mapsto t_{\bol{m}_i}\otimes x_i$
and $x_i^\ast \mapsto t_{-\bol{m}_i}\otimes x_i^\ast$, for some
$\bol{m}_i\in\bbZ^n$; in the notation above, we consider the free ${}^H\AAA$-algebra 
$F_{\bol{m}_1,\dots,\bol{m}_N,-\bol{m}_1,\dots,-\bol{m}_N}$ and
denote the last $N$ generators by $x_i^\ast := x_{N+i}$, for $i=1,\dots,N$.
The algebra of (the algebraic version of) the $2N{-}1$-dimensional Connes-Landi sphere is obtained by 
taking the quotient with respect to the ${}^H\AAA$-ideal
\begin{flalign}
I_{\mathbb{S}_\Theta^{2N-1}} := \Big(\, \mbox{$\sum\limits_{i=1}^N$}
\,  x_i^\ast\,x_i -\1\, \Big)~,
\end{flalign}
which implements the unit sphere relation. 
The algebra of the $N$-dimensional noncommutative torus is obtained by
taking the quotient with respect to the ${}^H\AAA$-ideal
\begin{flalign}
I_{\mathbb{T}_\Theta^N} := \big(x_i^\ast\,x_i -\1 \,:\, i=1,\dots,N\big)~.
\end{flalign} 
To obtain also the even dimensional Connes-Landi spheres, we consider the free ${}^H\AAA$-algebra 
$F_{\bol{m}_1,\dots,\bol{m}_N,-\bol{m}_1,\dots,-\bol{m}_N,\bol{0}}$, where the additional generator
$x_{2N+1}$ has trivial $H$-coaction $x_{2N+1}\mapsto \1_H\otimes x_{2N+1}$,
and take the quotient with respect to the ${}^H\AAA$-ideal
\begin{flalign}
I_{\mathbb{S}_\Theta^{2N}} := \Big(\, \mbox{$\sum\limits_{i=1}^N$} \,
x_i^\ast\,x_i + (x_{2N+1})^2 -\1\, \Big)~.
\end{flalign}
All these examples are $\ast$-algebras with involution defined by $x_i \mapsto x_i^\ast$
and $x_{2N+1}^\ast = x_{2N+1}$. An example which is not a $\ast$-algebra
is the free ${}^H\AAA$-algebra $F_{\bol{m}}$, for some
$\bol{m}\neq\bol{0}$ in $\bbZ^n$,
which we may interpret as the algebra of (anti)holomorphic polynomials on $\bbC$.
\end{ex}

We will now study some properties of the categories 
${}^H\AAA$ and ${}^H\AAA_{\mathrm{fp}}$ that will be used in 
the following. First, let us notice that the category ${}^H\AAA$ has (finite) coproducts:
Given two objects $A$ and $B$ in ${}^H\AAA$ their coproduct $A\sqcup B$
is the object in ${}^H\AAA$ whose underlying left $H$-comodule is $A\otimes B$ 
(with coaction $\rho^{A\sqcup B} := \rho^{A\otimes B}$ given in \eqref{eqn:tensorcoaction}) 
and whose product $\mu_{A\sqcup B}$ and unit $\eta_{A\sqcup B}$
are characterized  by
\begin{subequations}\label{eqn:tensoralgebra}
\begin{flalign}
(a\otimes b) \,(a^\prime\otimes b^\prime\, ) &:=
R(a_{(-1)}^\prime\otimes b_{(-1)})~(a\,a_{(0)}^\prime)\otimes
(b_{(0)}\,b^\prime\, )~,\\[4pt]
\1_{A\sqcup B} &:= \1_A^{}\otimes \1_B^{}~.
\end{flalign}
\end{subequations}
The canonical inclusion ${}^H\AAA$-morphisms $\iota_1^{} : A\to A\sqcup B$ and $\iota_2^{} : B\to A\sqcup B$
are given by
\begin{flalign}
\iota_{1}^{}(a) = a\otimes \1_B^{}\quad,\qquad \iota_2^{}(b) = \1_A^{}\otimes b~, 
\label{eqn:inclusions}\end{flalign}
for all $a\in A$ and $b\in B$. The coproduct $A\sqcup B$ 
of two finitely presented ${}^H\AAA$-algebras $A$ and $B$ is finitely presented:
If $A = F_{\bol{m}_1,\dots, \bol{m}_N}/(f_k)$ and 
$B =  F_{\bol{m^\prime}_1,\dots, \bol{m^\prime}_{N^\prime}}/(f^\prime_{k^\prime})$,
then 
\begin{flalign}
A\sqcup B \simeq F_{\bol{m}_1,\dots, \bol{m}_N,\bol{m^\prime}_1,\dots,
  \bol{m^\prime}_{N^\prime}}\, \big/\,
\big(f_k\otimes\1, \, \1\otimes f^\prime_{k^\prime}\big)~,
\end{flalign}
where we have identified $F_{\bol{m}_1,\dots, \bol{m}_N,\bol{m^\prime}_1,\dots, \bol{m^\prime}_{N^\prime}}\simeq
F_{\bol{m}_1,\dots, \bol{m}_N} \sqcup F_{\bol{m^\prime}_1,\dots, \bol{m^\prime}_{N^\prime}}$.
Consequently, the category ${}^H\AAA_{\mathrm{fp}}$ has finite coproducts.
\sk

In addition to coproducts, we also need pushouts in ${}^H\AAA$ and ${}^H\AAA_{\mathrm{fp}}$, which are given 
by colimits of the form
\begin{flalign}\label{eqn:pushout}
\xymatrix{
\ar[d]_-{\kappa}C\ar[r]^-{\zeta} & B\ar@{-->}[d]\\
A \ar@{-->}[r]& A\sqcup_C^{} B
}
\end{flalign}
in ${}^H\AAA$ or ${}^H\AAA_{\mathrm{fp}}$. Such pushouts exist and can be constructed as follows:
Consider first the case where we work in the category ${}^H\AAA$. We define
\begin{flalign}\label{eqn:pushoutformula}
A\sqcup_C^{} B := A\sqcup B / I~,
\end{flalign}
where $I$ is the ${}^H\AAA$-ideal generated by $\kappa(c)\otimes \1_B^{} - \1_{A}^{}\otimes\zeta(c)$, for all
$c\in C$. The dashed ${}^H\AAA$-morphisms in \eqref{eqn:pushout} are given by
\begin{flalign}\label{eqn:pullbackarrows}
A\longrightarrow A\sqcup_C^{} B~,~~a\longmapsto [a\otimes \1_B^{}]\quad,\qquad
B\longrightarrow A\sqcup_C^{} B~,~~b\longmapsto [\1_{A}^{}\otimes b]~.
\end{flalign}
It is easy to confirm that $A\sqcup_C^{}B$ defined above is a pushout of the diagram \eqref{eqn:pushout}.
Moreover, the pushout of finitely presented ${}^H\AAA$-algebras is finitely presented: 
If $A = F_{\bol{m}_1,\dots, \bol{m}_N}/(f_k)$,
$B =  F_{\bol{m^\prime}_1,\dots, \bol{m^\prime}_{N^\prime}}/(f^\prime_{k^\prime})$
and $C =  F_{\bol{m^{\prime\prime}}_1,\dots, \bol{m^{\prime\prime}}_{N^{\prime\prime}}}/
(f^{\prime\prime}_{k^{\prime\prime}})$, then
\begin{flalign}\label{eqn:pushoutformulafg}
A\sqcup_C^{} B \simeq F_{\bol{m}_1,\dots,
  \bol{m}_N,\bol{m^\prime}_1,\dots, \bol{m^\prime}_{N^\prime}}\, \big/\,
\big(f_k\otimes \1,\, \1\otimes f^\prime_{k^\prime}, \,
\kappa(x^{\prime\prime}_{i}\, )\otimes\1 - \1\otimes
\zeta(x^{\prime\prime}_i\, ) \big)~,
\end{flalign}
where $x^{\prime\prime}_i$, for $i=1,\dots, N^{\prime\prime}$, are the generators of  $C$.
The isomorphism in \eqref{eqn:pushoutformulafg} follows from the fact
that the quotient by the {\em finite} number of extra relations 
$\kappa(x^{\prime\prime}_{i}\, )\otimes\1 - \1\otimes
\zeta(x^{\prime\prime}_i\, )$,
for all generators $x^{\prime\prime}_i$ of $C$, is sufficient to describe the ${}^H\AAA$-ideal $I$ in
\eqref{eqn:pushoutformula} for finitely presented ${}^H\AAA$-algebras
$C$: We can recursively use the identities (valid on the right-hand
side of \eqref{eqn:pushoutformulafg})
\begin{subequations}
\begin{flalign}
\big[\kappa(x^{\prime\prime}_i\,c)\otimes\1 - \1\otimes \zeta(x^{\prime\prime}_i\,c)\big]
&= \big[\big(\kappa(x^{\prime\prime}_i\, )\otimes\1 \big) \,
\big(\kappa(c)\otimes\1 - \1 \otimes \zeta(c)\big)\big]~,\\[4pt]
\big[\kappa(c\, x^{\prime\prime}_i\, )\otimes\1 - \1\otimes
\zeta(c\,x^{\prime\prime}_i\, )\big]&=
 \big[\big(\kappa(c)\otimes\1 - \1\otimes \zeta(c)\big)\,
\big(\kappa(x^{\prime\prime}_i\, )\otimes\1 \big) \big]~,
\end{flalign}
\end{subequations}
for all generators $x^{\prime\prime}_i$ and elements $c\in C$, in order to show that the ${}^H\AAA$-ideal $I$ is 
equivalently generated by $\kappa(x^{\prime\prime}_{i}\,
)\otimes\1 - \1 \otimes \zeta(x^{\prime\prime}_i\, )$.
Consequently, the category ${}^H\AAA_{\mathrm{fp}}$ has pushouts.
\sk

We also need the localization of ${}^H\AAA$-algebras $A$
with respect to a single {\em $H$-coinvariant} element $s \in A$, i.e.\ $\rho^A(s) = \1_H^{}\otimes s$.
Localization amounts to constructing an 
${}^H\AAA$-algebra $A[s^{-1}]$ together with an ${}^H\AAA$-morphism
$\ell_s : A\to A[s^{-1}]$ that maps the element $s\in A$ to an invertible element $\ell_s(s)\in A[s^{-1}]$
and that satisfies the following universal property: If $\kappa : A\to B$
is another ${}^H\AAA$-morphism such that $\kappa(s)\in B$ is invertible,
then $\kappa$ factors though $\ell_s : A\to A[s^{-1}]$, i.e.\ there
exists a unique ${}^H\AAA$-morphism $A[s^{-1}]\to B$ making the diagram
\begin{flalign}
\xymatrix{
\ar[rd]_-{\ell_s}A\ar[r]^-{\kappa} & B\\
& A[s^{-1}]\ar@{-->}[u]
}
\end{flalign}
commute.
We now show that the ${}^H\AAA$-algebra
\begin{subequations}\label{eqn:localizationformula}
\begin{flalign}
A[s^{-1}] := A\sqcup F_{\bol{0}}/(s\otimes x- \1_{A\sqcup F_{\bol{0}}}^{})
\end{flalign}
 together with the ${}^H\AAA$-morphism 
\begin{flalign}
\ell_s : A\longrightarrow A[s^{-1}]~,~~a\longmapsto [a\otimes\1_{F_{\bol{0}}}^{}]
\end{flalign}
\end{subequations}
is a localization of $A$ with respect to the $H$-coinvariant element $s\in A$.
The inverse of $\ell_s(s) = [s\otimes \1_{F_{\bol{0}}}^{}]$ exists and it is given by the new generator
$[\1_A^{}\otimes x]\in A[s^{-1}]$; then the inverse of 
$\ell_s(s^n)$ is $[\1_A^{}\otimes x^n]$, because $[s\otimes \1_{F_{\bol{0}}}^{}]$
and $[\1_A^{}\otimes x]$ commute in $A[s^{-1}]$, cf.\ \eqref{eqn:commutationrelations}, \eqref{eqn:tensoralgebra}
and use the fact that $x$ and $s$ are coinvariants.
Given now any ${}^H\AAA$-morphism $\kappa : A\to B$ such that $\kappa(s)\in B$ is invertible, say by
$t\in B$, then there is a unique ${}^H\AAA$-morphism $A[s^{-1}]\to B$ specified
by $[a\otimes x^{n}]\mapsto \kappa(a)\, t^n$ that factors $\kappa$ through $\ell_s :A\to A[s^{-1}]$.
Finally, the localization of finitely presented ${}^H\AAA$-algebras is finitely presented: 
If $A= F_{\bol{m}_1,\dots,\bol{m}_N}/(f_k)$, then
\begin{flalign}
A[s^{-1}]\simeq F_{\bol{m}_1,\dots,\bol{m}_N,\bol{0}}/(f_k\otimes
\1 ,\, s\otimes x - \1 )~.
\end{flalign}

\section{\label{sec:spaces}Toric noncommutative spaces}
From a geometric perspective, it is useful to interpret an object 
$A$ in ${}^H\AAA_{\mathrm{fp}}$ as the `algebra of functions' on a toric noncommutative space $X_A$. 
Similarly, a morphism $\kappa : A\to B$ in ${}^H\AAA_{\mathrm{fp}}$
is interpreted as the `pullback' of a map $f : X_B\to X_A$ between toric noncommutative spaces,
where due to contravariance of pullbacks the direction of the arrow is reversed when going from algebras to spaces.
We shall use the more intuitive notation $\kappa = f^\ast : A\to B$ for the ${}^H\AAA_{\mathrm{fp}}$-morphism
corresponding to $f : X_B\to X_A$. This can be made precise with
\begin{defi}
The category of toric noncommutative spaces
\begin{flalign}
{}^H\SSS := \big({}^H\AAA_{\mathrm{fp}}\big)^\op
\end{flalign}
is the opposite of the category ${}^H\AAA_{\mathrm{fp}}$.
Objects in ${}^H\SSS$ will be denoted by symbols like $X_A$, where
$A$ is an object in ${}^H\AAA_{\mathrm{fp}}$. Morphisms
in ${}^H\SSS$ will be denoted by symbols like $f : X_{B}\to X_{A}$ and they are (by definition)
in bijection with ${}^H\AAA_{\mathrm{fp}}$-morphisms
$f^\ast : A\to B$.
\end{defi}

As the category ${}^H\AAA_{\mathrm{fp}}$ has (finite) coproducts and pushouts,
which we have denoted by $A\sqcup B$ and $A\sqcup_C^{} B$,
its opposite category ${}^H\SSS$ has (finite) products and pullbacks.
Given two objects $X_A$ and $X_B$ in ${}^H\SSS$, their
product is given by
\begin{subequations}
\begin{flalign}
X_A\times X_B := X_{A\sqcup B}~,
\end{flalign}
together with the canonical projection ${}^H\SSS$-morphisms
\begin{flalign}
\pi_{1}^{} : X_A\times X_B \longrightarrow X_A \quad,\qquad \pi_{2}^{} : X_A\times X_B \longrightarrow X_B
\end{flalign}
\end{subequations}
specified by $\pi_{1}^\ast = \iota_{1}^{} : A\to A\sqcup B$
and $\pi_{2}^\ast = \iota_{2}^{} : B\to A\sqcup B$, where
$\iota_{1}^{}$ and $\iota_{2}^{}$ are the canonical inclusion ${}^H\AAA_{\mathrm{fp}}$-morphisms
for the coproduct in ${}^H\AAA_{\mathrm{fp}}$ (cf.\ \eqref{eqn:inclusions}).
Pullbacks 
\begin{flalign}\label{eqn:pullback}
\xymatrix{
\ar@{-->}[d]X_A\times_{X_C}^{} X_B \ar@{-->}[r] & X_B\ar[d]^-{g}\\
X_A \ar[r]_-{f} & X_C
}
\end{flalign}
in ${}^H\SSS$ are given by
\begin{flalign}
X_A\times_{X_C}^{} X_B  := X_{A\sqcup_C^{} B}~,
\end{flalign}
for $\kappa = f^\ast : C\to A $ and $\zeta = g^\ast :C\to B $ (cf.\ \eqref{eqn:pushout}). 
The dashed arrows in \eqref{eqn:pullback} are specified by their
corresponding  ${}^H\AAA_{\mathrm{fp}}$-morphisms in \eqref{eqn:pullbackarrows}.
\sk

We next introduce a suitable notion of covering for toric noncommutative spaces,
which is motivated by the well-known Zariski covering families 
in commutative algebraic geometry.
\begin{defi}\label{def:zariski}
An ${}^H\SSS$-Zariski covering family is a finite family of ${}^H\SSS$-morphisms 
\begin{flalign}
\big\{f_{i} : X_{A[s_i^{-1}]} \longrightarrow X_{A}^{}\big\}~,
\end{flalign}
where
\begin{itemize}
\item[(i)] $s_i\in A$ is an $H$-coinvariant element, i.e.\ $\rho^A(s_i) =\1_H^{}\otimes s_i$, for all $i$;
\item[(ii)] $f_i$ is specified by the canonical ${}^H\AAA_{\mathrm{fp}}$-morphism
$f_i^\ast = \ell_{s_i} : A\to A[s_i^{-1}]$, for all $i$;
\item[(iii)] there exists a family of elements $a_i\in A$ such that $\sum_i\, a_i\,s_i =\1_A^{}$.
\end{itemize}
\end{defi}
\begin{ex}
Recall from Example \ref{ex:algebras} that the algebra of functions on
the $2N$-dimensional
Connes-Landi sphere is given by 
\begin{flalign}
A_{\mathbb{S}^{2N}_\Theta} = 
F_{\bol{m}_1,\dots,\bol{m}_N,-\bol{m}_1,\dots,-\bol{m}_N,\bol{0}}\,
\big/\, I_{\mathbb{S}_{\Theta}^{2N}}~.
\end{flalign}
As the last generator $x_{2N+1}$ is $H$-coinvariant, we can define the two $H$-coinvariant
elements $s_1 := \frac{1}{2}\, (\1- x_{2N+1})$ and $s_{2} := \frac{1}{2}\, (\1+ x_{2N+1})$.
Then $s_1 + s_2 =\1$ and hence we obtain an ${}^H\SSS$-Zariski covering family
\begin{flalign}
\Big\{ f_i : X_{A_{\mathbb{S}^{2N}_\Theta}[s_i^{-1}]} \longrightarrow X_{A_{\mathbb{S}^{2N}_\Theta}}^{}\Big\}_{i=1,2}
\end{flalign}
for the $2N$-dimensional Connes-Landi sphere. Geometrically, $X_{A_{\mathbb{S}^{2N}_\Theta}[s_1^{-1}]}$
is the sphere with the north pole removed and similarly $X_{A_{\mathbb{S}^{2N}_\Theta}[s_2^{-1}]}$
the sphere with the south pole removed. 
\end{ex}

We now show that ${}^H\SSS$-Zariski covering families are stable under pullbacks.
\begin{propo}\label{prop:pullbackstablecovers}
The pullback of an ${}^H\SSS$-Zariski covering family $\{f_{i} : X_{A[s_i^{-1}]} \to X_{A}^{}\}$ along
an ${}^H\SSS$-morphism $g : X_{B} \to X_A$ is an ${}^H\SSS$-Zariski
covering family, i.e. the left vertical arrows of the pullback diagrams
\begin{flalign}\label{eqn:tmppullbackstable}
\xymatrix{
\ar@{-->}[d] X_{B}\times_{X_A}^{} X_{A[s_i^{-1}]} \ar@{-->}[r] & X_{A[s_i^{-1}]}\ar[d]^-{f_i}\\
X_B\ar[r]_-{g}& X_A
}
\end{flalign}
define an ${}^H\SSS$-Zariski covering family.
\end{propo}
\begin{proof}
By definition, $X_{B}\times_{X_A}^{} X_{A[s_i^{-1}]} = X_{B\sqcup_A^{} A[s_i^{-1}]}$. 
By universality of the pushout and localization, the pushout diagram for 
$B\sqcup_A^{} A[s_i^{-1}]$ extends to the commutative diagram
\begin{flalign}
\xymatrix{
\ar[d]_-{g^\ast}A\ar[r]^-{\ell_{s_i}} & A[s_i^{-1}]\ar[d]\ar@/^1pc/[rdd]^-{} &\\
 \ar@/_1pc/[rrd]_-{\ell_{g^\ast(s_i)}}B \ar[r]& B\sqcup_A^{} A[s_i^{-1}] \ar@{-->}[dr]& \\
&& B[g^\ast(s_i)^{-1}]
}
\end{flalign}
It is an elementary computation to confirm that the dashed arrow in this diagram is an isomorphism
by using the explicit formulas for the pushout \eqref{eqn:pushoutformula} and localization \eqref{eqn:localizationformula}.
As a consequence, $X_{B}\times_{X_A}^{} X_{A[s_i^{-1}]} \simeq X_{B[g^\ast(s_i)^{-1}]}$
and the left vertical arrow in \eqref{eqn:tmppullbackstable} is of the form
as required in Definition \ref{def:zariski} (i) and (ii).
To show also item (iii) of Definition \ref{def:zariski}, if $a_i\in A$
is a family of elements
such that $\sum_i \, a_i\,s_i=\1_A^{}$, then $b_i := g^\ast(a_i)\in B$
is a family of elements
such that $\sum_{i}\, b_i \, g^{\ast}(s_i) = \1_B^{}$.
\end{proof}
\begin{cor}\label{cor:intersection}
Let $\{f_{i} : X_{A[s_i^{-1}]} \to X_{A}^{}\}$ be an ${}^H\SSS$-Zariski covering family.
Then the pullback 
\begin{flalign}\label{eqn:intersection}
\xymatrix{
\ar@{-->}[d]X_{A[s_i^{-1}]}\times_{X_A}^{} X_{A[s_j^{-1}]} \ar@{-->}[r] & X_{A[s_j^{-1}]}\ar[d]^-{f_{j}}\\
X_{A[s_i^{-1}]} \ar[r]_-{f_i} & X_{A}
}
\end{flalign}
is isomorphic to $X_{A[s_i^{-1},\, s_j^{-1}]}$, where 
\begin{flalign}
A[s_i^{-1},\, s_j^{-1}] := \big(A[s_i^{-1}]\big)[\ell_{s_i}(s_j)^{-1}] \simeq \big(A[s_j^{-1}]\big)[ \ell_{s_j}(s_i)^{-1}]
\end{flalign}
is the localization with respect to the two $H$-coinvariant elements $s_i,s_j\in A$.
The dashed arrows in \eqref{eqn:intersection} are specified by the canonical ${}^H\AAA_{\mathrm{fp}}$-morphisms
$\ell_{s_j} :   A[s_i^{-1}] \to A[s_i^{-1},\,s_j^{-1}]$ and
$\ell_{s_i} :   A[s_j^{-1}] \to A[s_i^{-1},\,s_j^{-1}]$.
\end{cor}
\begin{proof}
This follows immediately from the proof of Proposition \ref{prop:pullbackstablecovers}.
\end{proof}
\begin{rem}
For later convenience, we shall introduce the notation
\begin{flalign}\label{eqn:pullbacknotation}
\xymatrix{
\ar[d]_-{f_{j;i}} X_{A[s_i^{-1},s_j^{-1}]} \ar[r]^-{f_{i;j}} & X_{A[s_j^{-1}]}\ar[d]^-{f_{j}}\\
X_{A[s_i^{-1}]} \ar[r]_-{f_i} & X_{A}
}
\end{flalign}
for the morphisms of this pullback diagram.
\end{rem}

\section{\label{sec:genspaces}Generalized toric noncommutative spaces}
The category ${}^H\SSS$ of toric noncommutative spaces
has the problem that it does not generally admit exponential objects ${X_B}^{X_A}$, i.e.\
objects which describe the `mapping space' from $X_A$ to $X_B$. 
A similar problem is well-known from differential geometry, where the mapping 
space between two finite-dimensional manifolds in general is not 
a finite-dimensional manifold. There is however a canonical procedure for 
extending the category ${}^H\SSS$ to a bigger category that admits exponential objects.
We review this procedure in our case of interest.
\sk

The desired extension of ${}^H\SSS$ is given by the
category
\begin{flalign}
{}^H\GGG := \mathrm{Sh}\big({}^H\SSS\big)
\end{flalign}
of sheaves on ${}^H\SSS$ with covering families given in Definition \ref{def:zariski}.
Recall that a sheaf on ${}^H\SSS$ is a functor $Y : {}^H\SSS^\op \to \Set$ to the category of sets
(called a presheaf) that satisfies the sheaf condition: For any ${}^H\SSS$-Zariski covering family $\{f_i : X_{A[s_i^{-1}]} \to X_A^{}\}$
the canonical diagram
\begin{flalign}\label{eqn:sheafcondition}
\xymatrix{
Y\big(X_A\big) ~\ar[r] & ~\prod\limits_{i}\,  Y\big(X_{A[s_i^{-1}]}\big)  ~\ar@<-0.5ex>[r]\ar@<0.5ex>[r]& ~\prod\limits_{i,j}\,  Y\big(X_{A[s_i^{-1},\,s_j^{-1}]}\big)
}
\end{flalign}
is an equalizer in $\Set$. We have used Corollary \ref{cor:intersection} 
to express the pullback of two covering morphisms by $X_{A[s_i^{-1},\,s_j^{-1}]}$.
Because $ {}^H\SSS = ({}^H\AAA_{\mathrm{fp}})^\op$ was defined as the opposite 
category of ${}^H\AAA_{\mathrm{fp}}$, it is sometimes convenient to regard
a sheaf on ${}^H\SSS$ as a covariant functor $Y : {}^H\AAA_{\mathrm{fp}} \to \Set$.
In this notation, the sheaf condition \eqref{eqn:sheafcondition} looks like
\begin{flalign}\label{eqn:sheafconditionAlg}
\xymatrix{
Y(A) ~\ar[r] & ~\prod\limits_{i}\,  Y\big(A[s_i^{-1}]\big)  ~\ar@<-0.5ex>[r]\ar@<0.5ex>[r]& ~\prod\limits_{i,j}\,  Y\big(A[s_i^{-1},\,s_j^{-1}]\big)
}~.
\end{flalign}
We will interchangeably use these equivalent points of view.
The morphisms in ${}^H\GGG$ are natural transformations between functors,
i.e.\ presheaf morphisms.
\sk

We shall interpret ${}^H\GGG$ as a category of generalized toric noncommutative spaces.
To justify this interpretation, we will show that there is a fully faithful
embedding ${}^H\SSS \to {}^H\GGG$ of the category of toric noncommutative spaces
into the new category. As a first step, we use the (fully faithful) Yoneda embedding
${}^H\SSS \to \mathrm{PSh}({}^H\SSS)$ in order to embed ${}^H\SSS$
into the category of presheaves on ${}^H\SSS$. The Yoneda embedding is
given by the functor which assigns to any object $X_A$ in ${}^H\SSS$ 
the presheaf given by the functor
\begin{subequations}\label{eqn:Yoneda}
\begin{flalign}
\underline{X_A} := \Hom_{{}^H\SSS}^{}(-,X_A) : {}^H\SSS^\op \longrightarrow\Set
\end{flalign}
and to any ${}^H\SSS$-morphism $f : X_A\to X_B$ the natural transformation
$\underline{f}: \underline{X_A}\to \underline{X_B}$ with components
\begin{flalign}
\underline{f}_{X_C}^{} :  \Hom_{{}^H\SSS}^{}(X_C,X_A) \longrightarrow  \Hom_{{}^H\SSS}^{}(X_C,X_B)~,~~ 
g\longmapsto f\circ g~,
\end{flalign}
\end{subequations}
for all objects $X_C$ in ${}^H\SSS$.
\begin{propo}\label{propo:fullyfaithfulYoneda}
For any object $X_{A}$ in ${}^H\SSS$ the presheaf $\underline{X_A}$ is a sheaf on ${}^H\SSS$.
As a consequence, the Yoneda embedding induces a fully faithful embedding
${}^H\SSS\to {}^H\GGG$ into the category of sheaves on ${}^H\SSS$.
\end{propo}
\begin{proof}
We have to show that the functor $\underline{X_A}: {}^H\SSS^\op \to \Set$ 
satisfies the sheaf condition \eqref{eqn:sheafcondition}, 
or equivalently \eqref{eqn:sheafconditionAlg}.
Given any ${}^H\SSS$-Zariski covering family $\{f_i : X_{B[s_i^{-1}]}\to X_B^{}\}$,
we therefore have to confirm that
\begin{flalign}
\xymatrix{
\Hom_{{}^H\AAA_{\mathrm{fp}}}^{}\big(A,B\big) ~\ar[r] & ~\prod\limits_{i}\,  \Hom_{{}^H\AAA_{\mathrm{fp}}}^{}\big(A,B[s_i^{-1}]\big)  ~\ar@<-0.5ex>[r]\ar@<0.5ex>[r]& ~\prod\limits_{i,j}\,  \Hom_{{}^H\AAA_{\mathrm{fp}}}^{}\big(A,B[s_i^{-1},\,s_j^{-1}]\big)
}
\end{flalign}
is an equalizer in $\Set$, where we used $\underline{X_A}(X_B) = \Hom_{{}^H\SSS}^{}(X_B,X_A)
= \Hom_{{}^H\AAA_{\mathrm{fp}}}^{}(A,B)$. Because the $\Hom$-functor
$\Hom_{{}^H\AAA_{\mathrm{fp}}} (A,-) : {}^H\AAA_{\mathrm{fp}} \to \Set$ 
preserves limits, it is sufficient to prove that
\begin{flalign}\label{eqn:equalizerAAA}
\xymatrix{
B~\ar[r]& ~\prod\limits_{i}\, B[s_i^{-1}]
~\ar@<-0.5ex>[r]\ar@<0.5ex>[r]& ~\prod\limits_{i,j}\, B[s_i^{-1},\,s_j^{-1}]
}
\end{flalign}
is an equalizer in ${}^H\AAA_{\mathrm{fp}}$.
\sk

Using the explicit characterization of
localizations (cf.\ \eqref{eqn:localizationformula}), let us take a generic element
\begin{flalign}
\prod_{i}\, [b_i\otimes c_i] \ \in \ \prod\limits_{i}\, B[s_i^{-1}] =
\prod\limits_{i}\, B\sqcup F_{\bol{0}}/(s_i\otimes x_i - \1)~,
\end{flalign}
where here there is no sum over the index $i$ but an implicit sum of 
the form $[b_i\otimes c_i] = \sum_{\alpha}\, [{(b_i)}_{\alpha} \otimes {(c_{i})}_\alpha]$
which we suppress. This is an element in the desired equalizer if and only if
\begin{flalign}\label{eqn:tmpsubcanonical}
[b_i\otimes c_i\otimes \1]  = [b_j\otimes \1\otimes c_j]~,
\end{flalign}
for all $i,j$, as equalities in $B[s_{i}^{-1},\,s_j^{-1}]$. Recalling that the relations in $B[s_{i}^{-1},\,s_j^{-1}]$
are given by $s_i\otimes x_i \otimes \1 =\1$ and $s_j\otimes \1\otimes x_j=\1$,
the equalities \eqref{eqn:tmpsubcanonical} hold if and only if
$[b_i\otimes c_i] = [b\otimes\1]$ with the same $b\in B$, 
for all $i$. Hence \eqref{eqn:equalizerAAA} is an equalizer.
\end{proof}

\begin{rem}
Heuristically, Proposition \ref{propo:fullyfaithfulYoneda} implies that
the theory of toric noncommutative spaces $X_A$ together with their morphisms
can be equivalently described within the category ${}^H\GGG$.
The sheaf $\underline{X_A}$ specified by \eqref{eqn:Yoneda}
is interpreted as the `functor of points' of the toric noncommutative space $X_A$.
In this interpretation \eqref{eqn:Yoneda} tells us all possible ways in which
any other toric noncommutative space $X_B$ may be mapped into $X_A$,
which captures the geometric structure of $X_A$.
A generic object $Y$ in ${}^H\GGG$ (which we call a generalized toric noncommutative space)
has a similar interpretation: The set $Y(X_B)$ tells us all possible ways in which
$X_B$ is mapped into $Y$. This is formalized by Yoneda's Lemma
\begin{flalign}
Y(X_B) \simeq \Hom_{{}^H\GGG}(\, \underline{X_B}\, ,Y)~,
\end{flalign}
for any object $X_B$ in ${}^H\SSS$ and any object $Y$ in ${}^H\GGG$.
\end{rem}

The advantage of the sheaf category ${}^H\GGG$ of generalized toric noncommutative spaces
over the original category ${}^H\SSS$ of toric noncommutative spaces
is that it has very good categorical properties, which are summarized in the notion
of a Grothendieck topos, see e.g.\ \cite{MacLaneMoerdijk}.
Most important for us are the facts that ${}^H\GGG$ has all (small) limits and all exponential objects.
Limits in ${}^H\GGG$ are computed object-wise, i.e.\ as in presheaf categories.
In particular, the product of two objects $Y,Z$  in ${}^H\GGG$ is the sheaf
specified by the functor $Y\times Z : {}^H\SSS^{\op} \to \Set$ that acts on objects as
\begin{subequations}
\begin{flalign}
(Y\times Z)(X_A) := Y(X_A)\times Z(X_A)~,
\end{flalign}
where on the right-hand side $\times $ is the Cartesian product in $\Set$,
and on morphisms $ f : X_A\to X_B$ as
\begin{flalign}
(Y\times Z)(f) := Y(f)\times Z(f)\, :\, ( Y\times Z)(X_B)
\longrightarrow (Y\times Z) (X_A)~.
\end{flalign}
\end{subequations}
The terminal object in ${}^H\GGG$ is the sheaf specified by the functor
$\{\ast\} : {}^H\SSS^{\op} \to \Set $ that acts on objects as
\begin{flalign}
\{\ast\}(X_A) := \{\ast\}~,
\end{flalign}
where on the right-hand side $\{\ast\}$ is the terminal object in
$\Set$, i.e.\ a singleton set,
and in the obvious way on morphisms. 
The fully faithful embedding ${}^H\SSS\to {}^H\GGG$ 
of Proposition \ref{propo:fullyfaithfulYoneda} 
is limit-preserving. In particular, we have 
\begin{flalign}\label{eqn:productpreserving}
\underline{X_A\times X_B} = \underline{X_A}\times \underline{X_B}\quad,\qquad \underline{X_{\bbK}} = \{\ast\}~,
\end{flalign}
for all objects $X_A,X_B$ in ${}^H\SSS$ and the terminal object $X_{\bbK}$ in ${}^H\SSS$. Here $\bbK$ is
the ${}^H\AAA_{\mathrm{fp}}$-algebra with trivial left $H$-coaction $c\mapsto \1_H\otimes c$, i.e.\ the initial object
in ${}^H\AAA_{\mathrm{fp}}$.
\sk

The exponential objects in ${}^H\GGG$ are constructed as follows:
Given two objects $Y,Z$ in ${}^H\GGG$, the exponential object $Z^Y$ is the sheaf
specified by the functor $Z^Y : {}^H\SSS^{\op} \to \Set$ that acts on objects as
\begin{subequations}\label{eqn:exponentialobject}
\begin{flalign}
Z^Y(X_A) := \Hom_{{}^H\GGG}\big(\, \underline{X_A} \times Y, Z\big)~,
\end{flalign}  
and on morphisms $f : X_A\to X_B$ as
\begin{flalign}
Z^Y(f) : Z^Y(X_B)\longrightarrow Z^Y(X_A)~,~~g\longmapsto g\circ (\, \underline{f}\times \id_{Y})~,
\end{flalign}
\end{subequations}
where $\underline{f}:\underline{X_A}\to\underline{X_B}$ 
is the ${}^H\GGG$-morphism specified by \eqref{eqn:Yoneda}.
The formation of exponential objects is functorial, i.e.\ there are obvious 
functors $(-)^Y : {}^H\GGG\to {}^H\GGG$ and $Z^{(-)} : {}^H\GGG^\op\to {}^H\GGG$,
for all objects $Y,Z$ in ${}^H\GGG$. Moreover, there are natural isomorphisms
\begin{flalign}
\{\ast\}^Y\simeq\{\ast\} ~,~~~Z^{\{\ast\}}\simeq Z~,~~~(Z\times
Z^\prime\, )^Y\simeq Z^Y\times {Z^\prime}\,^Y~,~~~
Z^{Y\times Y^\prime} \simeq {(Z^Y)}^{Y^\prime}~,
\end{flalign}
for all objects $Y,Y^\prime,Z,Z^\prime$ in ${}^H\GGG$.
\sk

Given two ordinary toric noncommutative spaces $X_A$ and $X_B$, i.e.\ objects in ${}^H\SSS$,
we can form the exponential object $\underline{X_B}^{\underline{X_A}}$
in the category of {\em generalized} toric noncommutative spaces ${}^H\GGG$.
The interpretation of $\underline{X_B}^{\underline{X_A}}$ is as 
the `space of maps' from $X_A$ to $X_B$.
In the present situation, the explicit description \eqref{eqn:exponentialobject} 
of exponential objects may be simplified via
\begin{flalign}
\nn \underline{X_B}^{\underline{X_A}}(X_C) &= \Hom_{{}^H\GGG}\big(\,
\underline{X_C}\times \underline{X_A}\,,\,\underline{X_B}\, \big)\\[4pt]
\nn &= \Hom_{{}^H\GGG}\big(\, \underline{X_C\times
  X_A}\,,\,\underline{X_B}\, \big)\\[4pt]
\nn &\simeq \Hom_{{}^H\SSS}\big(X_C\times X_A,X_B\big) \\[4pt] &= \Hom_{{}^H\AAA_{\mathrm{fp}}}(B, C\sqcup A)~.
\end{flalign}
In the first step we have used \eqref{eqn:exponentialobject}, in the
second step \eqref{eqn:productpreserving}
and the third step is due to Yoneda's Lemma. Hence
\begin{flalign}\label{eqn:exponentialrepresentables}
\underline{X_B}^{\underline{X_A}} \simeq  \Hom_{{}^H\AAA_{\mathrm{fp}}}(B, - \sqcup A) 
: {}^H\AAA_{\mathrm{fp}}\longrightarrow\Set
\end{flalign}
can be expressed in terms of ${}^H\AAA_{\mathrm{fp}}$-morphisms.
\begin{ex}\label{ex:mappingline}
To illustrate the differences between the exponential objects $\underline{X_B}^{\underline{X_A}}$ in ${}^H\GGG$
and the $\Hom$-sets $\Hom_{{}^H\SSS}(X_A,X_B)$ let us consider the simplest example
where $A=B=F_{\bol{m}}$, for some $ \bol{m}\neq\bol{0}$ in $\bbZ^n$. In this case the $\Hom$-set is given by
\begin{flalign}
\Hom_{{}^H\SSS}(X_{F_{\bol{m}}},X_{F_{\bol{m}}}) = \Hom_{{}^H\AAA_{\mathrm{fp}}} (F_{\bol{m}},F_{\bol{m}}) \simeq \bbK~,
\end{flalign}
because by $H$-equivariance 
any ${}^H\AAA_{\mathrm{fp}}$-morphism $\kappa : F_{\bol{m}}\to F_{\bol{m}}$ is of the form
$\kappa(x) = c\,x$, for some $c\in \bbK$; here $x$ denotes the generator of $F_{\bol{m}}$, whose
left $H$-coaction is by definition $\rho^{F_{\bol{m}}}(x) = t_{\bol{m}}\otimes x$. On the other hand,
the exponential object $\underline{X_{F_{\bol{m}}}}^{\underline{X_{F_{\bol{m}}}}}$
is a functor from ${}^H\SSS^{\op} = {}^H\AAA_{\mathrm{fp}}$ to $\Set$ and hence it gives us
a set for any test ${}^H\AAA_{\mathrm{fp}}$-algebra $A$. Using \eqref{eqn:exponentialrepresentables}
we obtain
\begin{flalign}
\underline{X_{F_{\bol{m}}}}^{\underline{X_{F_{\bol{m}}}}}(A) = 
\Hom_{{}^H\AAA_{\mathrm{fp}}}(F_{\bol{m}}, A \sqcup F_{\bol{m}}) ~.
\end{flalign}
For the initial object $A = \bbK$ in ${}^H\AAA_{\mathrm{fp}}$, we recover
the $\Hom$-set
\begin{flalign}
\underline{X_{F_{\bol{m}}}}^{\underline{X_{F_{\bol{m}}}}}(\bbK) =
 \Hom_{{}^H\SSS}(X_{F_{\bol{m}}},X_{F_{\bol{m}}})\simeq \bbK \ .
\end{flalign}
Let us now take $A_{\bbT_{\Theta}} = F_{-\bol{m},\bol{m}}/(y^\ast\, y -\1)$ to be the toric 
noncommutative circle, see Example~\ref{ex:algebras}. We write $y^{-1}:=y^\ast $ in $A_{\bbT_{\Theta}}$
and recall that the left $H$-coaction is given by $\rho^{F_{-\bol{m},\bol{m}} }(y) = t_{-\bol{m}}\otimes y$
and $\rho^{F_{-\bol{m},\bol{m}} }(y^{-1}) = t_{\bol{m}}\otimes y^{-1}$.
We then obtain an isomorphism of sets
\begin{flalign}
\underline{X_{F_{\bol{m}}}}^{\underline{X_{F_{\bol{m}}}}}(A_{\bbT_{\Theta}})
= \Hom_{{}^H\AAA_{\mathrm{fp}}}(F_{\bol{m}},  A_{\bbT_{\Theta}}\sqcup F_{\bol{m}}) \simeq F_{\bol{m}}~,
\end{flalign}
because elements $a\in F_{\bol{m}}$, i.e.\ polynomials $a= \sum_{j}\, 
c_j\,x^j$, for $c_j\in\bbK$, 
are in bijection with ${}^H\AAA_{\mathrm{fp}}$-morphisms
$\kappa : F_{\bol{m}}\to A_{\bbT_{\Theta}} \sqcup F_{\bol{m}}$ via 
\begin{flalign}
\kappa(x) = \sum_{j}\, c_j ~ y^{j-1}\otimes x^j~.
\end{flalign}
By construction each summand on the right-hand side 
has left $H$-coaction $\rho^{F_{-\bol{m},\bol{m}} }(y^{j-1}\otimes x^j) = t_{-(j-1)\,\bol{m}}\,t_{j\,\bol{m}} \otimes y^{j-1}\otimes x^j
= t_{\bol{m}}\otimes y^{j-1}\otimes x^j$.
Heuristically, this means that the exponential object $\underline{X_{F_{\bol{m}}}}^{\underline{X_{F_{\bol{m}}}}}$
captures all polynomial maps $F_{\bol{m}}\to F_{\bol{m}}$, while the $\Hom$-set 
$\Hom_{{}^H\SSS}(X_{F_{\bol{m}}},X_{F_{\bol{m}}})$  captures only
those that are $H$-equivariant which in the present case are the linear maps $x\mapsto c\,x$.
Similar results hold for generic exponential objects $\underline{X_B}^{\underline{X_A}}$ in ${}^H\GGG$;
in particular, their global points $\underline{X_B}^{\underline{X_A}}(X_\bbK)$ coincide with the $\Hom$-sets
$\Hom_{{}^H\SSS}(X_A,X_B)$ while their generalized points $\underline{X_B}^{\underline{X_A}}(X_C)$, for
$X_C$ an object in ${}^H\SSS$, capture additional maps. 
\end{ex}

\section{\label{sec:automorphism}Automorphism groups}
Associated to any object $X_A$ in ${}^H\SSS$ is its exponential object
$\underline{X_A}^{\underline{X_A}}$ in ${}^H\GGG$ which describes
the `space of maps' from $X_A$ to itself.
The object $\underline{X_A}^{\underline{X_A}}$ has a distinguished
point (called the identity) given by the ${}^H\GGG$-morphism
\begin{subequations}\label{eqn:identity}
\begin{flalign}
e : \{\ast\} \longrightarrow \underline{X_A}^{\underline{X_A}}
\end{flalign}
that is specified by the natural transformation with components
\begin{flalign}
e_{X_B} : \{\ast\} \longrightarrow \Hom_{{}^H\SSS}\big(X_B\times X_A,X_A\big)~,~~\ast\longmapsto \big(\pi_{2} : 
X_B\times X_A \to X_A\big)~
\end{flalign}
\end{subequations}
given by the canonical projection ${}^H\SSS$-morphisms of the product.
Under the Yoneda bijections $\Hom_{{}^H\GGG}(\{\ast\}, \underline{X_A}^{\underline{X_A}}) \simeq 
\underline{X_A}^{\underline{X_A}}(\{\ast\})\simeq \Hom_{{}^H\SSS}(X_A,X_A)$, $e$ is mapped
to the identity ${}^H\SSS$-morphism $\id_{X_A}$. 
Moreover, there is a composition ${}^H\GGG$-morphism
\begin{subequations}\label{eqn:composition}
\begin{flalign}
\bullet : \underline{X_A}^{\underline{X_A}} \times \underline{X_A}^{\underline{X_A}}\longrightarrow \underline{X_A}^{\underline{X_A}}
\end{flalign}
that is specified by the natural transformation with components
\begin{flalign}
\nn \bullet_{X_B} : \Hom_{{}^H\SSS}\big(X_B\times X_A,X_A\big)\times \Hom_{{}^H\SSS}\big(X_B\times X_A,X_A\big)&\longrightarrow 
\Hom_{{}^H\SSS}\big(X_B\times X_A,X_A\big)~,\\
 (g,h) & \longmapsto g\bullet_{X_B} h
\end{flalign}
defined by
\begin{flalign}
g\bullet_{X_B} h := g\circ (\id_{X_B}\times h)\circ (\mathrm{diag}_{X_B}\times\id_{X_A}) : X_B\times X_A\longrightarrow X_A~,
\end{flalign}
\end{subequations}
for all ${}^H\SSS$-morphisms $g,h : X_B\times X_A\to X_A$.
The diagonal ${}^H\SSS$-morphism $\mathrm{diag}_{X_B} : X_B\to X_B\times X_B$
is defined as usual via universality of products by
\begin{flalign}
\xymatrix@C=4.0em{
& \ar[dl]_-{\id_{X_B}}X_B \ar@{-->}[d]^-{\mathrm{diag}_{X_B}}\ar[dr]^-{\id_{X_B}}& \\
X_B &\ar[l]^-{\pi_1} X_B\times X_B \ar[r]_-{\pi_2}& X_B
}
\end{flalign}
The object $ \underline{X_A}^{\underline{X_A}}$ together with the identity \eqref{eqn:identity}
and composition ${}^H\GGG$-morphisms \eqref{eqn:composition} is a monoid object
in ${}^H\GGG$: It is straightforward to verify that the ${}^H\GGG$-diagrams
\begin{subequations}
\begin{flalign}
\xymatrix{
\ar[d]_-{\id\times\bullet}\underline{X_A}^{\underline{X_A}} \times \underline{X_A}^{\underline{X_A}}\times \underline{X_A}^{\underline{X_A}}\ar[rr]^-{\bullet\times \id}&& \underline{X_A}^{\underline{X_A}}\times \underline{X_A}^{\underline{X_A}}\ar[d]^-{\bullet}\\
\underline{X_A}^{\underline{X_A}}\times \underline{X_A}^{\underline{X_A}}\ar[rr]_-{\bullet} && \underline{X_A}^{\underline{X_A}}
}
\end{flalign}
and
\begin{flalign}
\xymatrix{
\ar[drr]_-{\simeq }\{\ast\} \times \underline{X_A}^{\underline{X_A}} \ar[rr]^-{e\times\id} &&\underline{X_A}^{\underline{X_A}}\times\underline{X_A}^{\underline{X_A}}\ar[d]^-{\bullet} && \ar[ll]_-{\id\times e}\underline{X_A}^{\underline{X_A}}\times\{\ast\}\ar[dll]^-{\simeq}\\
&& \underline{X_A}^{\underline{X_A}} &&
}
\end{flalign}
\end{subequations}
commute.
Notice that $\underline{X_A}^{\underline{X_A}}$ is not a group object in ${}^H\GGG$
because, loosely speaking, generic maps do not have an inverse. We may however construct
the `subobject of invertible maps' (in a suitable sense to be detailed below) 
of the monoid object $\underline{X_A}^{\underline{X_A}}\,$, 
which then becomes a group object in ${}^H\GGG$ called 
the automorphism group $\Aut(X_A)$ of $X_A$.
\sk

Let us apply the fully faithful functor $\mathrm{Sh}({}^H\SSS) \to \mathrm{PSh}({}^H\SSS)$
(which assigns to sheaves their underlying presheaves) 
on the monoid object $(\, \underline{X_A}^{\underline{X_A}}\, ,\bullet, e)$ in ${}^H\GGG = \mathrm{Sh}({}^H\SSS)$
to obtain the monoid object $(\, \underline{X_A}^{\underline{X_A}}\, ,\bullet, e)$ in $\mathrm{PSh}({}^H\SSS)$,
denoted with abuse of notation by the same symbol.
This monoid object in $\mathrm{PSh}({}^H\SSS)$ may be equivalently regarded
as a functor ${}^H\SSS^{\op} \to \mathsf{Monoid}$ with values in the category of ordinary $\Set$-valued monoids 
(i.e.\ monoid objects in the category $\Set$). The functor assigns to
any object $X_B$ in ${}^H\SSS$ the monoid
\begin{subequations}\label{eqn:stagewisemonoid}
\begin{flalign}
\Big(\underline{X_A}^{\underline{X_A}}(X_B) , \bullet_{X_B} , e_{X_B}\Big) 
\end{flalign}
and to any ${}^H\SSS$-morphism $f: X_B\to X_C$ the monoid morphism
\begin{flalign}
\underline{X_A}^{\underline{X_A}}(f)\, :\, \Big(\underline{X_A}^{\underline{X_A}}(X_C) , \bullet_{X_C} , e_{X_C}\Big)  \longrightarrow
\Big(\underline{X_A}^{\underline{X_A}}(X_B) , \bullet_{X_B} , e_{X_B}\Big) ~.
\end{flalign}
\end{subequations}
For any object $X_B$ in ${}^H\SSS$, we define $\Aut(X_A)(X_B)$ to be the
subset of elements $g\in \underline{X_A}^{\underline{X_A}}(X_B)$ for
which there exists $g^{-1}\in \underline{X_A}^{\underline{X_A}}(X_B)$
such that
\begin{flalign}\label{eqn:autdefi}
g\bullet_{X_B} g^{-1} = g^{-1}\bullet_{X_B} g = e_{X_B} \ .
\end{flalign}
Because the inverse of an element in a monoid (if it exists) is always unique, 
it follows that any element $g\in \Aut(X_A)(X_B)$ has a unique inverse $g^{-1}\in \Aut(X_A)(X_B)$,
and that the inverse of $g^{-1}$ is $g$. The monoid structure on $\underline{X_A}^{\underline{X_A}}(X_B)$
induces to $\Aut(X_A)(X_B)$, because the inverse of $e_{X_B}$ is $e_{X_B}$ itself
and the inverse of $g\bullet_{X_B} h$ is $(g\bullet_{X_B} h)^{-1} = h^{-1}\bullet_{X_B} g^{-1}$.
Denoting by
\begin{flalign}
\mathrm{inv}_{X_B} : \Aut(X_A)(X_B)\longrightarrow \Aut(X_A)(X_B)~,~~g\longmapsto g^{-1}
\end{flalign}
the map that assigns the inverse, we obtain for any object $X_B$ in ${}^H\SSS$
a group 
\begin{subequations}\label{eqn:stagewisegroup}
\begin{flalign}
\big(\Aut(X_A)(X_B), \bullet_{X_B}, e_{X_B},\mathrm{inv}_{X_B}\big)~.
\end{flalign}
The monoid morphism $\underline{X_A}^{\underline{X_A}}(f)$ in \eqref{eqn:stagewisemonoid}
induces a group morphism which we denote by
\begin{flalign}
 \resizebox{0.91\textwidth}{!}{$\Aut(X_A)(f)\,:\,\big(\Aut(X_A)(X_C), \bullet_{X_C}, e_{X_C},\mathrm{inv}_{X_C}\big)\longrightarrow
\big(\Aut(X_A)(X_B), \bullet_{X_B}, e_{X_B},\mathrm{inv}_{X_B}\big)~.$}
\end{flalign}
\end{subequations}
Hence we have constructed a functor ${}^H\SSS^{\op} \to \mathsf{Group}$ with values 
in the category of ordinary $\Set$-valued groups (i.e.\ group objects in $\Set$), which we
can equivalently regard as a group object $(\Aut(X_A),\bullet,e,\mathrm{inv})$ 
in the category $\mathrm{PSh}({}^H\SSS)$. Notice further 
that $\Aut(X_A)$ is a subobject of $\underline{X_A}^{\underline{X_A}}$
in the category $\mathrm{PSh}({}^H\SSS)$. 
\begin{propo}\label{propo:autgroup}
For any object $X_A$ in ${}^H\SSS$, the presheaf $\Aut(X_A)$
satisfies the sheaf condition \eqref{eqn:sheafcondition}. In particular,
$(\Aut(X_A),\bullet,e,\mathrm{inv})$ is the subobject of invertibles 
of the monoid object $(\underline{X_A}^{\underline{X_A}},\bullet,e)$ in ${}^H\GGG$
and hence a group object in ${}^H\GGG$.
\end{propo}
\begin{proof}
Given any ${}^H\SSS$-Zariski covering family $\{f_i : X_{B[s_i^{-1}]} \to X_{B}\}$,
we have to show that
\begin{flalign}
\xymatrix{
\Aut(X_A)\big(X_B\big) ~\ar[r] & ~\prod\limits_{i}\,  \Aut(X_A)\big(X_{B[s_i^{-1}]}\big)  ~\ar@<-0.5ex>[r]\ar@<0.5ex>[r]& ~\prod\limits_{i,j}\,  \Aut(X_A)\big(X_{B[s_i^{-1},\,s_j^{-1}]}\big)~
}
\end{flalign}
is an equalizer in $\Set$. Recalling that $\Aut(X_A)$ is the sub-presheaf of the sheaf $\underline{X_A}^{\underline{X_A}}$
specified by the invertibility conditions \eqref{eqn:autdefi}, an element
in $\prod_{i}\,  \Aut(X_A)(X_{B[s_i^{-1}]})$ can be represented by an element
\begin{flalign}
\prod_{i} \, g_i \ \in \ \prod\limits_{i}\,
\underline{X_A}^{\underline{X_A}} \big(X_{B[s_i^{-1}]} \big)~,
\end{flalign}
such that each $g_i$ has an inverse 
$g_i^{-1}\in \underline{X_A}^{\underline{X_A}}\big(X_{B[s_i^{-1}]} \big)$ 
in the sense that 
\begin{flalign}
g_i\bullet_{X_{B[s_i^{-1}]}} g_i^{-1} = g_i^{-1}\bullet_{X_{B[s_i^{-1}]}} g_i= e_{X_{B[s_i^{-1}]}}~.
\end{flalign}
This element is in the desired equalizer if and only if
\begin{flalign}\label{eqn:tmpmatching}
\underline{X_A}^{\underline{X_A}}(f_{i;j})\big( g_j\big) = \underline{X_A}^{\underline{X_A}}(f_{j;i})\big( g_i\big) ~,
\end{flalign}
for all $i,j$, where we used the compact notation $f_{i;j}$ introduced in \eqref{eqn:pullbacknotation}.
Because $\underline{X_A}^{\underline{X_A}}$ is a sheaf, we can
represent $\prod_{i} \, g_i$
by the element $g\in \underline{X_A}^{\underline{X_A}}(X_B)$
that is uniquely specified by $\underline{X_A}^{\underline{X_A}}(f_i)(g) = g_i$, for all $i$.
\sk

We have to show that $g\in \Aut(X_A)(X_B)\subseteq \underline{X_A}^{\underline{X_A}}(X_B)$, i.e.\ that there exists
$g^{-1}\in \underline{X_A}^{\underline{X_A}}(X_B)$ such that $g\bullet_{X_B}g^{-1}= g^{-1}\bullet_{X_B} g = e_{X_B}$.
Since $\underline{X_A}^{\underline{X_A}}(f_{i;j})$ and $\underline{X_A}^{\underline{X_A}}(f_{j;i})$ are monoid morphisms,
both sides of the equality \eqref{eqn:tmpmatching} are invertible and the inverse is given by
\begin{flalign}
\underline{X_A}^{\underline{X_A}}(f_{i;j})\big( g_j^{-1}\big) = \underline{X_A}^{\underline{X_A}}(f_{j;i})\big( g_i^{-1}\big) ~.
\end{flalign}
Using again the
property that $\underline{X_A}^{\underline{X_A}}$ is a sheaf,
we can represent $\prod_{i} \, g^{-1}_i$ by the element $\tilde{g} \in \underline{X_A}^{\underline{X_A}}(X_B) $
that is uniquely specified by $\underline{X_A}^{\underline{X_A}}(f_i)(\tilde{g}) = g^{-1}_i$, for all $i$.
It is now easy to check that $\tilde{g}$ is the inverse of $g$:
Using once more the property that $\underline{X_A}^{\underline{X_A}}(f_i)$ are monoid morphisms,
we obtain $\underline{X_A}^{\underline{X_A}}(f_i)(\tilde{g}\bullet_{X_B} g)= g_i^{-1}\bullet_{X_{B[s_i^{-1}]}} g_i 
= e_{X_{B[s_i^{-1}]}}$ and similarly 
$\underline{X_A}^{\underline{X_A}}(f_i)(g\bullet_{X_B} \tilde{g})= e_{X_{B[s_i^{-1}]}}$, for all $i$.
Because  $\underline{X_A}^{\underline{X_A}}$ is a sheaf, this implies
$\tilde{g}\bullet_{X_B} g = g\bullet_{X_B} \tilde{g} = e_{X_B}$ and hence that $\tilde{g} = g^{-1}$.
\end{proof}

\section{\label{sec:liealgebras}Lie algebras of automorphism groups}
The category ${}^H\GGG$ of generalized toric noncommutative spaces 
has a distinguished object $\underline{K}:=\underline{X_{F_{\bol{0}}}}$, where $F_{\bol{0}}$ 
is the free ${}^H\AAA_{\mathrm{fp}}$-algebra with one coinvariant generator $x$, i.e.\ $x\mapsto \1_H\otimes x$. 
We call $\underline{K}$ the line object as it describes the toric noncommutative line.
The line object $\underline{K}$ is a ring object in ${}^H\GGG$: The sum ${}^H\GGG$-morphism
$+ : \underline{K}\times \underline{K}\to \underline{K}$ is induced (via going opposite and the Yoneda embedding)
by the ${}^H\AAA_{\mathrm{fp}}$-morphism $F_{\bol{0}} \to F_{\bol{0}}\sqcup F_{\bol{0}}\,,~x\mapsto x\otimes \1 + \1\otimes x$.
The multiplication ${}^H\GGG$-morphism $\cdot : \underline{K}\times \underline{K}\to \underline{K}$ is induced
by the ${}^H\AAA_{\mathrm{fp}}$-morphism $F_{\bol{0}} \to F_{\bol{0}}\sqcup F_{\bol{0}}\,,~x\mapsto x\otimes x$. 
The (additive) zero element is
the ${}^H\GGG$-morphism $0:\{\ast\}\to \underline{K}$ induced by the 
${}^H\AAA_{\mathrm{fp}}$-morphism $F_{\bol{0}}\to \bbK\,,~x\mapsto 0\in\bbK$,
and the (multiplicative) unit element is the ${}^H\GGG$-morphism $1:\{\ast\}\to \underline{K}$ 
induced by the  ${}^H\AAA_{\mathrm{fp}}$-morphism $F_{\bol{0}}\to \bbK\,,~x\mapsto 1\in\bbK$.
Finally, the additive inverse ${}^H\GGG$-morphism $ \mathrm{inv}_{+} : \underline{K}\to \underline{K}$ is induced
by the ${}^H\AAA_{\mathrm{fp}}$-morphism $F_{\bol{0}}\to F_{\bol{0}}\,,~x\mapsto -x$.
It is straightforward, but slightly tedious, to confirm that these structures make $\underline{K}$ into a ring object 
in ${}^H\GGG$.
\begin{rem}\label{rem:ringfunctor}
Regarding the line object as a functor $\underline{K} : {}^H\SSS^\op \to \Set$, 
it assigns to an object $X_B$ in ${}^H\SSS$ the set 
\begin{subequations}
\begin{flalign}
\underline{K}(X_B)= \Hom_{{}^H\SSS}(X_B, K) = \Hom_{{}^H\AAA_{\mathrm{fp}}}(F_{\bol{0}}, B) \simeq B^{\bol{0}}~,
\end{flalign}
where $B^{\bol{0}} := \{b\in B \,:\, \rho^B(b)=\1_H\otimes b \}$ is
the set of coinvariants; in the last step we have used the fact that $F_{\bol{0}}$ is the free ${}^H\AAA_{\mathrm{fp}}$-algebra with one coinvariant generator,
hence ${}^H\AAA_{\mathrm{fp}}$-morphisms $F_{\bol{0}} \to B$ are in bijection with $B^{\bol{0}}$.
To an ${}^H\SSS$-morphism $f : X_B\to X_{B^\prime}$ it assigns the restriction of $f^\ast : B^\prime\to B$ 
to coinvariants, i.e.\
\begin{flalign}
\underline{K}(f) = f^\ast : {B^\prime}\,^{\bol{0}} \longrightarrow B^{\bol{0}}~.
\end{flalign}
\end{subequations}
The ${}^H\AAA_{\mathrm{fp}}$-algebra structure on $B,B^\prime$ induces a (commutative) ring structure on
$B^{\bol{0}},{B^\prime}\,^{\bol{0}}$ and $f^\ast$ preserves this ring
structure. Hence we have obtained a functor
${}^H\SSS^{\op}\to\mathsf{CRing}$ with values in the category of commutative rings (in $\Set$), which is
an equivalent way to describe the ring object structure on $\underline{K}$ introduced above.
\end{rem}

The ${}^H\AAA_{\mathrm{fp}}$-morphism $F_{\bol{0}} \to F_{\bol{0}}/(x^2)$ given by the quotient map
induces a monomorphism $\underline{D} := \underline{X_{F_{\bol{0}}/(x^2)}} \to \underline{X_{F_{\bol{0}}}}= \underline{K}$
in ${}^H\GGG$. The zero element, sum and additive inverse of $\underline{K}$
induce to $\underline{D}$, i.e.\ we obtain ${}^H\GGG$-morphisms 
$0: \{\ast\}\to \underline{D}$, $+ : \underline{D}\times \underline{D}\to \underline{D}$
and $\mathrm{inv}_{+}:\underline{D}\to \underline{D}$ which give
$\underline{D}$ the structure of an Abelian group object in ${}^H\GGG$. 
Moreover, $\underline{D}$ is a $\underline{K}$-module object in ${}^H\GGG$
with left $\underline{K}$-action ${}^H\GGG$-morphism $\cdot : \underline{K}\times \underline{D}\to \underline{D}$
induced by the  ${}^H\AAA_{\mathrm{fp}}$-morphism 
$F_{\bol{0}}/(x^2) \to F_{\bol{0}}\sqcup F_{\bol{0}}/(x^2) \,,~ x\mapsto x\otimes x$.
Heuristically, $\underline{D}$ describes the infinitesimal neighborhood of $0$ in $\underline{K}$,
i.e.\ $\underline{D}$ is an infinitesimally short line, so short that functions on $\underline{D}$ 
(which are described by $F_{\bol{0}}/(x^2)$) are polynomials of degree 1. 
\sk

Following the ideas of synthetic (differential) geometry \cite{Kock,Lavendhomme,MoerdijkReyes},
we may use $\underline{D}$ to define the tangent bundle of a generalized toric noncommutative space.
\begin{defi}
Let $Y$ be any object in ${}^H\GGG$. The (total space of) the tangent bundle of $Y$ is
the exponential object
\begin{subequations}
\begin{flalign}
TY := Y^{\underline{D}}
\end{flalign}
in ${}^H\GGG$. The projection ${}^H\GGG$-morphism is
\begin{flalign}
\pi := Y^{0: \{\ast\} \to \underline{D}} \,: \, TY = Y^{\underline{D}} \longrightarrow Y^{\{\ast\}} \simeq Y~,
\end{flalign}
\end{subequations}
where we use the property that $Y^{(-)} : {}^H\GGG^\op \to {}^H\GGG$ is a functor.
\end{defi}
\begin{rem}\label{rem:functorperspectiveT}
Focusing on the underlying functors $Y : {}^H\SSS^\op\to \Set$
of objects $Y$ in ${}^H\GGG$, there is an equivalent, and practically useful,
characterization of the tangent bundle $TY$.
By Yoneda's Lemma and the fact that the exponential $(-)^{\underline{D}} : {}^H\GGG\to{}^H\GGG$ 
is the right adjoint of the product $-\times \underline{D}: {}^H\GGG\to {}^H\GGG$,
there are natural isomorphisms
\begin{flalign}
TY = Y^{\underline{D}} \simeq \Hom_{{}^H\GGG}\big(\, \underline{(-)}\,
, Y^{\underline{D}}\big)\simeq 
\Hom_{{}^H\GGG}\big(\, \underline{(-)}\times \underline{D}\, , Y\big)
\simeq Y\big(-\times D\big)~.
\end{flalign}
Hence the set $TY(X_B)$ at stage $X_B$ is simply given by $Y(X_B\times D)$ at stage $X_B\times D$.
The components of the projection then read as
\begin{flalign}
\pi_{X_B} = Y(\id_{X_B}\times 0)\,:\, Y(X_B\times D) \longrightarrow Y(X_B)~,
\end{flalign}
where $\id_{X_B}\times 0 : X_B\simeq  X_B\times\{\ast\} \to X_B\times D$ is the product of
the ${}^H\SSS$-morphisms $\id_{X_B} : X_B\to X_B$ and 
$0:\{\ast\} \to D$.
\end{rem}

We shall now study in more detail the tangent bundle
\begin{flalign}
\pi : T\Aut(X_A) \longrightarrow \Aut(X_A)~
\end{flalign}
of the automorphism group of some object $X_A$ in ${}^H\SSS$.
We are particularly interested in the fibre $T_e\Aut(X_A)$ 
of this bundle over the identity $e :\{\ast\} \to \Aut(X_A) $, because 
it defines the Lie algebra of $\Aut(X_A) $.
The fibre $T_e\Aut(X_A)$ is defined as the pullback
\begin{flalign}\label{eqn:tangentfibre}
\xymatrix{
\ar@{-->}[d]T_e\Aut(X_A) \ar@{-->}[rr] && T\Aut(X_A)\ar[d]^-{\pi}\\
\{\ast\}\ar[rr]_-{e} && \Aut(X_A)
}
\end{flalign}
in ${}^H\GGG$. In particular, $T_e\Aut(X_A)$ is an object in ${}^H\GGG$.
\sk

Using the perspective explained in Remark \ref{rem:functorperspectiveT},
we obtain 
\begin{flalign}
T\Aut(X_A)(X_B)  = \Aut(X_A)\big(X_B\times D\big)~,
\end{flalign}
for all objects $X_B$ in ${}^H\SSS$. 
The pullback \eqref{eqn:tangentfibre} then introduces a further condition
\begin{flalign}\label{eqn:autLie}
T_e\Aut(X_A)(X_B) = \left\{ g\in \Aut(X_A)\big(X_B\times D\big)\, :\, \Aut(X_A)(\id_{X_B}\times 0) (g) = e_{X_{B}}\right\}~,
\end{flalign}
for all objects $X_B$ in ${}^H\SSS$.
Using \eqref{eqn:autdefi} and \eqref{eqn:exponentialrepresentables}, 
it follows that any $g\in \Aut(X_A)(X_B\times D)$ is
an ${}^H\AAA_{\mathrm{fp}}$-morphism $A\to B\sqcup F_{\bol{0}}/(x^2) \sqcup A$
satisfying the invertibility condition imposed in \eqref{eqn:autdefi}.
For our purposes it is more convenient to equivalently  regard $g$ as an 
${}^H\AAA_{\mathrm{fp}}$-morphism $g : A\to F_{\bol{0}}/(x^2)\sqcup  B \sqcup A$
with target $ F_{\bol{0}}/(x^2)\sqcup  B \sqcup A$ instead of $B\sqcup F_{\bol{0}}/(x^2) \sqcup A$
(flipping $F_{\bol{0}}/(x^2)$ and $B$ is the usual flip map because the left $H$-coaction on $F_{\bol{0}}/(x^2)$
is trivial). Because any element in $F_{\bol{0}}/(x^2)$ is of the form $c_0 + c_1\,x$, for some $c_0,c_1\in\bbK$,
we obtain two ${}^H\MMM$-morphisms $g_0,g_1 : A\to B\sqcup A$ which are characterized uniquely by
\begin{flalign}\label{eqn:gsplitting}
g(a) =\1_{F_{\bol{0}}/(x^2)}\otimes g_0(a) + x\otimes g_1(a)~,
\end{flalign}
for all $a\in A$. Since $g : A\to  F_{\bol{0}}/(x^2)\sqcup B \sqcup A$ is an ${}^H\AAA_{\mathrm{fp}}$-morphism,
it follows that $g_0 : A\to B\sqcup A$ is an ${}^H\AAA_{\mathrm{fp}}$-morphism
and that $g_1 : A\to B\sqcup A$ satisfies the condition
\begin{flalign}\label{eqn:Leibniztmp}
g_1(a\,a^\prime\, ) = g_1(a)\,g_0(a^\prime\, ) +
g_0(a)\,g_1(a^\prime\, )~,
\end{flalign}
for all $a,a^\prime\in A$.
From \eqref{eqn:autLie} it follows that $g\in T_e\Aut(X_A)(X_B)$
if and only if $g_0 = \iota_2 : A\to B\sqcup A\,,~a\mapsto \1_B\otimes a$ 
is the canonical inclusion for the coproduct.
Then \eqref{eqn:Leibniztmp} simplifies to
\begin{flalign}\label{eqn:Leibniz}
g_1(a\,a^\prime\, ) = g_1(a)\,(\1_B\otimes a^\prime\, )+ (\1_B\otimes
a)\, g_1(a^\prime\, )~,
\end{flalign}
for all $a,a^\prime\in A$. Notice that \eqref{eqn:Leibniz} is the Leibniz rule
for the $A$-bimodule structure on $B\sqcup A$ that is induced
by the ${}^H\AAA_{\mathrm{fp}}$-algebra structure on $B\sqcup A$
and the inclusion ${}^H\AAA_{\mathrm{fp}}$-morphism $\iota_2 : A\to B\sqcup A$.
\begin{lem}\label{lem:invertibility}
Let $g : A\to  F_{\bol{0}}/(x^2) \sqcup B \sqcup A$ be any ${}^H\AAA_{\mathrm{fp}}$-morphism
such that $g_0 = \iota_2 : A\to B\sqcup A$ in the notation of \eqref{eqn:gsplitting}.
Then the  ${}^H\AAA_{\mathrm{fp}}$-morphism 
$\tilde{g} :  A\to   F_{\bol{0}}/(x^2)\sqcup B \sqcup A$ defined by 
$\tilde{g}_0 = g_0 = \iota_2$ and $\tilde{g}_1 = -g_1$ is the inverse of $g$
in the sense that $g\bullet_{X_B\times D} \tilde{g} = \tilde{g}\bullet_{X_B\times D} g = e_{X_B\times D}$.
\end{lem}
\begin{proof}
From the hypothesis, \eqref{eqn:composition} and $x^2=0$ it follows that
\begin{flalign}
 \resizebox{0.91\textwidth}{!}{($g\bullet_{X_B\times D} \tilde{g})\big(a\big) = \1_{F_{\bol{0}}/(x^2)}\otimes \1_{B} \otimes a  
+ x\otimes \big(g_{1}(a) + \tilde{g}_1(a)\big) = \1_{F_{\bol{0}}/(x^2)}\otimes \1_{B} \otimes a  = e_{X_B\times D}(a)~,$}
\end{flalign}
for all $a\in A$, and similarly that $\tilde{g}\bullet_{X_B\times D} g = e_{X_B\times D}$.
\end{proof}

This result allows us to give a very explicit characterization of the functor
underlying the object $T_e\Aut(X_A)$ in ${}^H\GGG$. Let us define
the functor ${}^H\Der(A,-\sqcup A): {}^H\SSS^\op \to \Set$
on objects $X_B$ by
\begin{subequations}\label{eqn:HDer}
\begin{flalign}
 {}^H\Der(A,-\sqcup A)\big(X_B\big) := {}^H\Der(A,B\sqcup A) \ ,
\end{flalign}
which is the subset of $v\in \Hom_{{}^H\MMM}(A,B\sqcup A)$ satisfying \eqref{eqn:Leibniz}, and on morphisms $f : X_B\to X_{B^\prime}$ by
\begin{flalign}
\nn {}^H\Der(A,-\sqcup A)\big(f\big)\, :\, {}^H\Der(A,B^\prime\sqcup A)&\longrightarrow {}^H\Der(A,B\sqcup A)~,\\
\big(v : A\to B^\prime \sqcup A\big)&\longmapsto\big((f^\ast\otimes \id_{A})\circ v : A\to B \sqcup A\big)~.
\end{flalign}
\end{subequations}
\begin{cor}\label{cor:DerAutIso}
The presheaf underlying $T_e\Aut(X_A)$ is isomorphic
to ${}^H\Der(A,-\sqcup A)$ via the natural isomorphism
with components
\begin{flalign}\label{eqn:DerLieiso}
T_e\Aut(X_A)\big(X_B\big)\longrightarrow {}^H\Der(A,B\sqcup A)~,~~g\longmapsto g_1~,
\end{flalign}
where $g_1$ is defined according to \eqref{eqn:gsplitting}. Hence
${}^H\Der(A,-\sqcup A)$ is a sheaf, i.e.\ an object in ${}^H\GGG$,
and $T_e\Aut(X_A)$ is also isomorphic to ${}^H\Der(A,-\sqcup A)$ in ${}^H\GGG$.
\end{cor}
\begin{proof}
Since $g$ is uniquely specified by $g_0,g_1$ (via \eqref{eqn:gsplitting})
and $g_0 = \iota_2$ for all $g\in T_e\Aut(X_A)(X_B)$, it follows that \eqref{eqn:DerLieiso}
is injective. Surjectivity of \eqref{eqn:DerLieiso} follows from Lemma~\ref{lem:invertibility}.
Naturality of \eqref{eqn:DerLieiso} is obvious. Because $T_e\Aut(X_A)$ and
${}^H\Der(A,-\sqcup A)$ are isomorphic as presheaves and $T_e\Aut(X_A)$ is
a sheaf, it follows from the fully faithful embedding
$\mathrm{Sh}({}^H\SSS)\to \mathrm{PSh}({}^H\SSS)$ 
that ${}^H\Der(A,-\sqcup A)$ is a sheaf and that the isomorphism is in ${}^H\GGG$.
\end{proof}

We conclude by showing that ${}^H\Der(A,-\sqcup A)$ (and hence by Corollary 
\ref{cor:DerAutIso} also $T_e\Aut(X_A)$) is a $\underline{K}$-module object in ${}^H\GGG$
that can be equipped with a Lie bracket. 
From the perspective used in Remark \ref{rem:ringfunctor}, 
this is equivalent to equipping ${}^H\Der(A,B\sqcup A)$ with a $B^{\bol{0}}$-module
structure and a Lie bracket on this $B^{\bol{0}}$-module,
such that both structures are natural transformations for 
${}^H\SSS$-morphisms $f : X_{B}\to X_{B^\prime}$.
Recall that ${}^H\Der(A,B\sqcup A)$ is the subset of 
$\Hom_{{}^H\MMM}(A,B\sqcup A)$ specified by the Leibniz rule 
\eqref{eqn:Leibniz}. 
Because the Leibniz rule is a linear condition, it follows that
${}^H\Der(A,B\sqcup A)$ is closed under taking sums and
additive inverses, and that it contains the zero map. From \eqref{eqn:HDer} one immediately
sees that this Abelian group structure is natural with respect to 
${}^H\SSS$-morphisms $f : X_{B}\to X_{B^\prime}$,
hence ${}^H\Der(A,-\sqcup A)$ is an Abelian group object in ${}^H\GGG$.
The $B^{\bol{0}}$-module structure
\begin{subequations}\label{eqn:B0ModHDer}
\begin{flalign}
B^{\bol{0}}\times {}^H\Der(A,B\sqcup A)\longrightarrow {}^H\Der(A,B\sqcup A)~,~~(b,v)\longmapsto b\cdot v
\end{flalign}
is defined by setting
\begin{flalign}
(b\cdot v) (a) := (b\otimes \1_A)~v(a)~,
\end{flalign}
\end{subequations}
for all $a\in A$. In order to verify that $(b\cdot v)\in {}^H\Der(A,B\sqcup A)$, i.e.\ that it is $H$-equivariant and
satisfies the Leibniz rule \eqref{eqn:Leibniz}, it is essential to use
the fact that $b$ is coinvariant, $\rho^B: b\mapsto \1_H\otimes b$.
From \eqref{eqn:HDer} one immediately sees that this $B^{\bol{0}}$-module structure
is natural with respect to ${}^H\SSS$-morphisms $f : X_{B}\to X_{B^\prime}$, i.e.\
\begin{flalign}
(f^\ast\otimes \id_{A})\circ (b^\prime\cdot v^\prime\, )
= f^\ast(b^\prime\, )\cdot \big((f^\ast\otimes \id_{A})\circ v^\prime\, \big)~,
\end{flalign}
for all $b^\prime\in {B^\prime}\,^{\bol{0}}$ and $v^\prime\in {}^H\Der(A,B^\prime\sqcup A)$.
This endows ${}^H\Der(A,-\sqcup A)$ with the structure of a $\underline{K}$-module object in ${}^H\GGG$.
\sk

It remains to define a Lie bracket
\begin{subequations}\label{eqn:Liebracket}
\begin{flalign}
[\,-\,,\,-\,]_{X_B}^{} \,:\, {}^H\Der(A,B\sqcup A)\otimes_{B^{\bol{0}}}{}^H\Der(A,B\sqcup A)\longrightarrow {}^H\Der(A,B\sqcup A)
\end{flalign}
on each $B^{\bol{0}}$-module ${}^H\Der(A,B\sqcup A)$.
Let us set
\begin{flalign}
[v,w]_{X_B}^{} := (\mu_B\otimes\id_A)\circ \big( (\id_B\otimes v)\circ w - 
 (\id_B\otimes w)\circ v\big)~,
\end{flalign}
\end{subequations}
for all $v,w\in {}^H\Der(A,B\sqcup A)$, where $\mu_B : B\otimes B\to B$ is the product on $B$.
Notice that $[v,w]_{X_B} : A\to B\sqcup A$ is an ${}^H\MMM$-morphism.
A straightforward but slightly lengthy computation (using the Leibniz rule \eqref{eqn:Leibniz} for $v$ and $w$) 
shows that $[v,w]_{X_B}$ satisfies the Leibniz rule, hence it is an element in  ${}^H\Der(A,B\sqcup A)$.
Antisymmetry of $[\,-\,,\,-\,]_{X_B}$ follows immediately from the definition and the Jacobi identity is 
shown by direct computation. Moreover, $B^{\bol{0}}$-linearity of the Lie bracket, i.e.\ 
\begin{flalign}
[b\cdot v,w]_{X_B} = b\cdot [v,w]_{X_B}= [v,b\cdot w]_{X_B}~,
\end{flalign}
for all $b\in B^{\bol{0}}$ and $v,w\in {}^H\Der(A,B\sqcup A)$,
can be easily verified by using the fact that $b$ is coinvariant and hence it commutes with any other element in $B$ 
(cf.\ \eqref{eqn:commutationrelations}).
Naturality of the Lie bracket with respect to ${}^H\SSS$-morphisms $f : X_{B}\to X_{B^\prime}$
is a simple consequence of the fact that $f^\ast : B^\prime \to B$ preserves the
products entering the definition in \eqref{eqn:Liebracket}. We have
thereby obtained
an explicit description of the Lie algebra of the automorphism group $\Aut(X_A)$.
\begin{propo}\label{propo:Liealgebra}
The functor ${}^H\Der(A,-\sqcup A)$ equipped with the structure morphisms introduced above is a Lie algebra object
in the category $\Mod_{\underline{K}}({}^H\GGG)$ of $\underline{K}$-module objects in ${}^H\GGG$. 
\end{propo}

\section{\label{sec:comparison}Braided derivations}
The Lie algebra object ${}^H\Der(A,-\sqcup A)$ constructed 
in Proposition \ref{propo:Liealgebra} is (isomorphic to) the Lie algebra of 
the automorphism group $\Aut(X_A)$. Hence we may interpret it
as the Lie algebra of infinitesimal automorphisms of the toric noncommutative space
$X_A$ with function ${}^H\AAA_{\mathrm{fp}}$-algebra $A$.
Another (a priori unrelated) way to think about the infinitesimal automorphisms
of $X_A$ is to consider the Lie algebra $\der(A)$ of braided derivations of $A$, see
\cite{NCgrav,NClect,BSS1,BSS2,AS}. In this section we show that these two points of view
are equivalent. 
\sk

We briefly introduce the concept of braided derivations 
of ${}^H\AAA_{\mathrm{fp}}$-algebras $A$.
Let us first consider the case where $A= F_{\bol{m}_1,\dots,\bol{m}_{N}}$
is the free ${}^H\AAA$-algebra with $N$ generators
$x_i$ with left $H$-coaction $x_i\mapsto t_{\bol{m}_i}\otimes x_i$,
for $i=1,\dots,N$.
Let $\partial_j : F_{\bol{m}_1,\dots,\bol{m}_{N}} \to F_{\bol{m}_1,\dots,\bol{m}_{N}}$, for $j=1,\dots,N$,
be the linear map defined by
\begin{subequations}
\begin{flalign}
\partial_j(x_i) &= \delta_{ij}\, \1 ~,\\[4pt]
\partial_j(a\,a^\prime\, ) &= \partial_j(a)\,a^\prime + R(a_{(-1)}\otimes t_{-\bol{m}_j}) ~a_{(0)}
\,\partial_j(a^\prime\, )~,
\end{flalign}
\end{subequations}
for all $i=1,\dots,N$ and $a,a^\prime\in F_{\bol{m}_1,\dots,\bol{m}_{N}} $. The map $\partial_j$
should be interpreted as the `partial derivative' along the generator $x_j$, hence it is natural to
assign to it the left $H$-coaction $\partial_j \mapsto t_{-\bol{m}_j}\otimes \partial_j$.
It satisfies a braided generalization of the Leibniz rule that is controlled 
by the cotriangular structure $R$. Let us define the left $H$-comodule
\begin{flalign}
\der(F_{\bol{m}_1,\dots,\bol{m}_N}) := \coprod_{j=1,\dots,N} \, F_{\bol{m}_1,\dots,\bol{m}_N}[-\bol{m}_j]~,
\end{flalign}
with the coproduct taken in ${}^H\MMM$, where for an object $V$ in ${}^H\MMM$ we denote by $V[\bol{m}]$ the object
in ${}^H\MMM$  which has the same underlying vector space as $V$ but which is 
equipped with the shifted left $H$-coaction
\begin{flalign}
\rho^{V[\bol{m}]} : V\longrightarrow H\otimes V~,~~v\longmapsto v_{(-1)}\,t_{\bol{m}} \otimes v_{(0)}~.
\end{flalign}
We denote elements $L \in \der(F_{\bol{m}_1,\dots,\bol{m}_N})$ by $L=\sum_j\, L_j\,\partial_j$,
where $L_j\in F_{\bol{m}_1,\dots,\bol{m}_N}$, because the $H$-coaction then takes the convenient form
$\rho^{\der(F_{\bol{m}_1,\dots,\bol{m}_N})}(L) =\sum_j\, {L_j}_{(-1)}\,t_{-\bol{m}_j} \otimes {L_{j}}_{(0)} \,\partial_j$.
The evaluation of $\der(F_{\bol{m}_1,\dots,\bol{m}_N})$ on $F_{\bol{m}_1,\dots,\bol{m}_N}$ is given by the ${}^H\MMM$-morphism
\begin{flalign}\label{eqn:dereval}
\ev : \der(F_{\bol{m}_1,\dots,\bol{m}_N}) \otimes F_{\bol{m}_1,\dots,\bol{m}_N}\longrightarrow F_{\bol{m}_1,\dots,\bol{m}_N}~,~~L\otimes a \longmapsto \sum_{j}\, L_j\, \partial_j(a)~.
\end{flalign}
It is then easy to confirm the braided Leibniz rule
\begin{flalign}\label{eqn:braidedLeibniz}
\ev\big(L\otimes (a\,a^\prime\, )\big) = \ev(L\otimes a)\,a^\prime +
R(a_{(-1)}\otimes L_{(-1)})~a_{(0)} \,\ev(L_{(0)}\otimes a^\prime\, )~,
\end{flalign}
for all $L\in \der(F_{\bol{m}_1,\dots,\bol{m}_N})$ and $a,a^\prime\in F_{\bol{m}_1,\dots,\bol{m}_N}$,
which allows us to interpret elements of $\der(F_{\bol{m}_1,\dots,\bol{m}_N})$ as braided derivations.
\sk

For a finitely presented ${}^H\AAA$-algebra $A = F_{\bol{m}_1,\dots,\bol{m}_N}/(f_k)$,
the left $H$-comodule of braided derivations is defined by
\begin{flalign}\label{eqn:derf.g.}
\der(A) := \Big\{  L \in \mbox{$\coprod\limits_{j=1,\dots,N}$}\,
A[-\bol{m}_j]  \,:\, \mbox{$\sum\limits_j$}\, L_j\,\partial_j(f_k)=
0~~\forall k\, \Big\} \ \subseteq \ \coprod_{j=1,\dots,N}\, A[-\bol{m}_j] ~.
\end{flalign}
The evaluation ${}^H\MMM$-morphism is similar to that in the case of free ${}^H\AAA$-algebras
and is given by
\begin{flalign}
\ev : \der(A) \otimes A \longrightarrow A~,~~L\otimes a \longmapsto \sum_{j}\, L_j\, \partial_j(a)~.
\end{flalign}
Notice that $\ev$ is well-defined because of the 
conditions imposed in \eqref{eqn:derf.g.}. The braided Leibniz rule \eqref{eqn:braidedLeibniz}
also holds in the case of finitely presented ${}^H\AAA$-algebras.
\begin{propo}
Let $A$ be any object in ${}^H\AAA_{\mathrm{fp}}$. Then
$\der(A)$ is a Lie algebra object in ${}^H\MMM$ with Lie bracket ${}^H\MMM$-morphism
$[\,-\,,\,-\,] : \der(A)\otimes\der(A) \to \der(A)$ uniquely defined by
\begin{flalign}\label{eqn:Liebracketderimplicit}
\ev\big([L,L^\prime\, ]\otimes a\big) := \ev\big(L\otimes\ev(L^\prime\otimes a)\big) -
 R(L^\prime_{(-1)}\otimes L_{(-1)})~\ev\big(L^\prime_{(0)}\otimes\ev(L_{(0)}\otimes a)\big) ~,
\end{flalign}
for all $L,L^\prime\in\der(A)$ and $a\in A$.
\end{propo}
\begin{proof}
Using the braided Leibniz rule \eqref{eqn:braidedLeibniz}, we can compute the right-hand side of \eqref{eqn:Liebracketderimplicit}
and obtain
\begin{flalign}
 \resizebox{0.91\textwidth}{!}{$
\ev\big([L,L^\prime\, ]\otimes a\big)  =
\sum\limits_{j,k} \, \big( L_j\,\partial_j(L^\prime_k) -R({L^\prime_j}_{(-1)}\,t_{-\bol{m}_j}\otimes {L_k}_{(-1)}\,\,t_{-\bol{m}_k})~{L^\prime_j}_{(0)}\,\partial_j({L_k}_{(0)})\big)\, \partial_k(a)~.$}
\end{flalign}
Hence $[L,L^\prime\,]$ 
is uniquely defined and it is a braided derivation. The braided antisymmetry and 
Jacobi identity on $[\,-\,,\,-\,]$ can be verified by a straightforward computation.
\end{proof}

We would now like to compare the braided derivations $\der(A)$ with the 
automorphism Lie algebra ${}^H\Der(A,-\sqcup A)$ constructed 
in Section \ref{sec:liealgebras}. There is however a problem:
While $\der(A)$ is a Lie algebra object in the category ${}^H\MMM$ of left $H$-comodules,
${}^H\Der(A,-\sqcup A)$ is a Lie algebra object in the category $\Mod_{\underline{K}}({}^H\GGG)$ 
of $\underline{K}$-module objects in the category ${}^H\GGG$ 
of generalized toric noncommutative spaces. We show in Appendix \ref{app:technical} that
there exists a functor $j : {}^H\MMM \to \Mod_{\underline{K}}({}^H\GGG)$,
which becomes a fully faithful embedding when restricted to the full subcategory
${}^H\MMM_{\mathrm{dec}}$ of decomposable left $H$-comodules 
(cf.\ Definition \ref{def:decomposable}). Because $\der(A)$ is a decomposable left $H$-comodule
(cf.\ Corollary \ref{cor:derisdecomposable}), we may use the fully faithful embedding
$j : {}^H\MMM_{\mathrm{dec}} \to \Mod_{\underline{K}}({}^H\GGG)$ to relate
$\der(A)$ to ${}^H\Der(A,-\sqcup A)$.
\sk

Let us characterize more explicitly the object $j(\der(A))$ in $\Mod_{\underline{K}}({}^H\GGG)$.
Its underlying functor (cf.\ \eqref{eqn:jfunctor}) 
assigns to an object $X_B$ in ${}^H\SSS$ the $B^\bol{0}$-module
\begin{subequations}
\begin{flalign}
j(\der(A))(X_B) = \big(B\otimes \der(A)\big)^{\bol{0}}
\end{flalign}
and to an ${}^H\SSS$-morphism $f : X_B\to X_C$ the module morphism
\begin{flalign}
j(\der(A))(f) = (f^\ast\otimes \id_{\der(A)}) \, :\,  \big(C\otimes \der(A)\big)^{\bol{0}}\longrightarrow \big(B\otimes \der(A)\big)^{\bol{0}}~.
\end{flalign}
\end{subequations}
For any object $X_B$ in ${}^H\SSS$ we define a map
\begin{subequations}\label{eqn:zetamapping}
\begin{flalign}
\xi_{X_B} : \big(B\otimes \der(A)\big)^{\bol{0}} \longrightarrow {}^H\Der(A,B\sqcup A)~,~~
b\otimes L \longmapsto \xi_{X_B} (b\otimes L)
\end{flalign}
by setting
\begin{flalign}
\xi_{X_B}(b\otimes L)(a) := b\otimes \ev(L\otimes a) =  b\otimes \Big(\sum_j\, L_j\,\partial_j(a)\Big)~,
\end{flalign}
\end{subequations}
for all $a\in A$. It is easy to check that $\xi_{X_B}(b\otimes L) : A\to B\sqcup A$
is an ${}^H\MMM$-morphism by using the property that $b\otimes L$ is $H$-coinvariant. 
Moreover, $\xi_{X_B}(b\otimes L)$ satisfies the Leibniz rule \eqref{eqn:Leibniz}
because $L$ satisfies the braided Leibniz rule \eqref{eqn:braidedLeibniz} and $b\otimes L$ is $H$-coinvariant.
Explicitly we have
\begin{flalign}
\nn \xi_{X_B}(b\otimes L)(a\,a^\prime\, ) &= 
b\otimes\left(\ev(L\otimes a)\,a^\prime + R(a_{(-1)}\otimes
  L_{(-1)})\, a_{(0)}\, \ev(L_{(0)}\otimes a^\prime\, )\right)\\[4pt]
&=\xi_{X_B}(b\otimes L)(a)\,(\1_B\otimes a^\prime\, ) + (\1_B\otimes
a)\, \xi_{X_B}(b\otimes L)(a^\prime\, )~,
\end{flalign}
where the last step follows from \eqref{eqn:tensoralgebra} and \eqref{eqn:Rmatrixproperties}.
This shows that the image of $\xi_{X_B}$ lies in ${}^H\Der(A,B\sqcup A)$, 
as we have asserted in \eqref{eqn:zetamapping}.
\sk

The maps $\xi_{X_B}$ are clearly $B^{\bol{0}}$-module morphisms with respect to the $B^\bol{0}$-module
structure on ${}^H\Der(A,B\sqcup A)$ introduced in \eqref{eqn:B0ModHDer} and that on $j(\der(A))(X_B)$ 
introduced in \eqref{eqn:B0ModjV}, and they are natural with respect to ${}^H\SSS$-morphisms
$f : X_B\to X_C$. Hence we have defined a morphism
\begin{flalign}\label{eqn:zetanattrafo}
\xi :  j(\der(A)) \longrightarrow {}^H\Der(A,-\sqcup A)~
\end{flalign}
in the category $\Mod_{\underline{K}}({}^H\GGG)$ of $\underline{K}$-module objects in ${}^H\GGG$.
The main result of this section is
\begin{theo}\label{theo:dervsHDer}
The $\Mod_{\underline{K}}({}^H\GGG)$-morphism \eqref{eqn:zetanattrafo} is an isomorphism. Hence
$\der(A)$ and ${}^H\Der(A,-\sqcup A)$ are equivalent descriptions of the infinitesimal automorphisms
of a toric noncommutative space $X_A$.
\end{theo}
\begin{proof}
Fix a presentation $A = F_{\bol{m}_1,\dots,\bol{m}_N}/(f_k)$ of $A$ 
and any object $X_B$ in ${}^H\SSS$.
We define {\em non-equivariant} linear maps $\widehat{\partial}_j : 
F_{\bol{m}_1,\dots,\bol{m}_N} \to B\otimes F_{\bol{m}_1,\dots,\bol{m}_N}$
by setting
\begin{subequations}
\begin{flalign}
\widehat{\partial}_j(x_i) &=\delta_{ij}~ \1_B\otimes \1_{F_{\bol{m}_1,\dots,\bol{m}_N}}~,\\[4pt]
\widehat{\partial}_j(a\,a^\prime\, ) &= \widehat{\partial}_j(a)\, (\1_B\otimes a) + 
R(a_{(-1)}\otimes t_{-\bol{m}_j})~(\1_B\otimes a_{(0)}) \,
\widehat{\partial}_j(a^\prime\, )~,
\end{flalign}
\end{subequations}
for all generators $x_i$ and all $a,a^\prime \in F_{\bol{m}_1,\dots,\bol{m}_N}$.
There is an isomorphism
\begin{flalign}\label{eqn:conditionHDer}
\Big\{v\in \mbox{$\coprod\limits_{j=1,\dots,N}\, (B\otimes A[-\bol{m}_j])^{\bol{0}} \,:\, 
\sum\limits_j\, v_j \, \widehat{\partial}_j(f_k) =0~~\forall k$}\,\Big\} \simeq
{}^H\Der(A,B\sqcup A) 
\end{flalign}
given by the assignment $v \mapsto \sum_j\,
v_j\,\widehat{\partial}_j$. Because $A$ and $B$ are decomposable,
we obtain a chain of isomorphisms
\begin{flalign}
\nn \coprod_{j=1,\dots,N} (B\otimes A[-\bol{m}_j])^{\bol{0}}&\simeq
\coprod_{j=1,\dots,N} \ \coprod_{\bol{n}\in\bbZ^n} \,
\big(B^{\bol{n}}\otimes A[-\bol{m}_j]^{-\bol{n}}\big) \\[4pt]
\nn &\simeq  \coprod_{\bol{n}\in\bbZ^n} \, \Big(B^{\bol{n}}\otimes
\Big(\coprod_{j=1,\dots,N} \, A[-\bol{m}_j]\Big)^{-\bol{n}}\Big)
\\[4pt] &\simeq
\Big(B\otimes \coprod_{j=1,\dots,N}\, A[-\bol{m}_j]\Big)^{\bol{0}}~.
\end{flalign}
The resulting isomorphism preserves the conditions imposed in \eqref{eqn:conditionHDer} and \eqref{eqn:derf.g.},
hence it induces an isomorphism between ${}^H\Der(A,B\sqcup A)$ and $j(\der(A))(X_B)$.
\end{proof}
\begin{rem}
Even though the functor $j : {}^H\MMM\to \Mod_{\underline{K}}({}^H\GGG)$
is not monoidal (cf.\ Remark~\ref{rem:notmonoidal}),
there exists a $\Mod_{\underline{K}}({}^H\GGG)$-morphism
\begin{flalign}
\psi : j(\der(A))\otimes j(\der(A)) \longrightarrow j(\der(A)\otimes\der(A))~,
\end{flalign}
which is described  explicitly in \eqref{eqn:psimap}. We now confirm that the isomorphism
$\xi: j(\der(A)) \to {}^H\Der(A,-\sqcup A)$ established in Theorem \ref{theo:dervsHDer} preserves the Lie brackets
on $\der(A)$ and ${}^H\Der(A,-\sqcup A)$ in the sense that the diagram
\begin{flalign}
\xymatrix{
\ar[dd]_-{\xi\otimes \xi} j(\der(A))\otimes j(\der(A)) \ar[rr]^-{\psi} &&  j(\der(A)\otimes\der(A))\ar[d]^-{j([\,-\,,\,-\,])}\\
&&  j(\der(A))\ar[d]^-{\xi}\\
{}^H\Der(A,-\sqcup A)\otimes {}^H\Der(A,-\sqcup A) \ar[rr]_-{[\,-\,,\,-\,]} && {}^H\Der(A,-\sqcup A)
}
\end{flalign}
in $\Mod_{\underline{K}}({}^H\GGG)$ commutes. Fixing an arbitrary object $X_B\in {}^H\SSS$ and
going along the upper path of this diagram we obtain
\begin{subequations}\label{eqn:Liecomparison}
\begin{flalign}
\nn &\left(\xi_{X_B}\circ (\id_{B}\otimes [\,-\,,\,-\,])\circ
  \psi_{X_B}\big((b\otimes L) \otimes_{B^{\bol{0}}} (b^\prime\otimes
  L^\prime\, )\big)\right)(a)\\
\nn &\qquad ~\qquad =R(b^\prime_{(-1)}\otimes L_{(-1)})
~b\,b^\prime_{(0)}\otimes \ev\big([L_{(0)},L^\prime\, ]\otimes a\big)\\[4pt]
\nn &\qquad ~\qquad  = R(b^\prime_{{(-1)}_{(1)}}\otimes L_{(-1)})
~R(b^\prime_{{(-1)}_{(2)}}\otimes b_{(-1)})~b^\prime_{(0)}\,b_{(0)}
\otimes \ev\big([L_{(0)},L^\prime\, ]\otimes a\big)\\[4pt]
\nn &\qquad ~\qquad  = R(b^\prime_{{(-1)}}\otimes b_{(-1)}\, L_{(-1)})
~b^\prime_{(0)}\,b_{(0)} \otimes \ev\big([L_{(0)},L^\prime\, ]\otimes a\big)\\[4pt]
&\qquad ~\qquad = b^\prime\,b\otimes \ev\big([L,L^\prime\, ]\otimes a\big)~,\label{eqn:Liecomparisonupper}
\end{flalign}
for all $a\in A$, where in the last two steps we used the properties \eqref{eqn:Rmatrixproperties} of the cotriangular structure $R$
and the fact that $b\otimes L \in (B\otimes\der(A))^{\bol{0}}$ is
coinvariant. Going now along the lower path of the diagram
we obtain
\begin{flalign}\label{eqn:Liecomparisonlower}
\big[\xi_{X_B}(b\otimes L), \xi_{X_B}(b^\prime\otimes L^\prime\, )\big]_{X_B}(a)= 
b^\prime \,b\otimes \ev\big(L\otimes \ev(L^\prime\otimes a)\big) - b\,b^\prime \otimes 
\ev\big(L^\prime\otimes \ev(L\otimes a)\big)~,
\end{flalign}
\end{subequations}
for all $a\in A$, where we used the definition of the Lie bracket $ [\,-\,,\,-\,]_{X_B} $ given in \eqref{eqn:Liebracket}.
These two expressions coincide because, using without loss of generality $b\otimes L\in B^{\bol{m}}\otimes \der(A)^{-\bol{m}}$
and $b^\prime\otimes L^\prime\in B^{\bol{m}^\prime}\otimes \der(A)^{-\bol{m}^\prime}$,
the second term in  \eqref{eqn:Liecomparisonlower} can be rearranged as
\begin{flalign}
\nn b\,b^\prime \otimes \ev\big(L^\prime\otimes \ev(L\otimes a)\big) &= R(b^\prime_{(-1)}\otimes b_{(-1)})~ b^{\prime}_{(0)}\,b_{(0)} \otimes \ev\big(L^\prime\otimes \ev(L\otimes a)\big)\\[4pt]
\nn &= R(t_{\bol{m}^\prime}\otimes t_{\bol{m}})~ b^{\prime}\,b \otimes \ev\big(L^\prime\otimes \ev(L\otimes a)\big)\\[4pt]
\nn &= R(t_{-\bol{m}^\prime}\otimes t_{-\bol{m}})~ b^{\prime}\,b \otimes \ev\big(L^\prime\otimes \ev(L\otimes a)\big)\\[4pt]
  &= R(L_{(-1)}^\prime \otimes L_{(-1)})~ b^{\prime}\,b \otimes \ev\big(L_{(0)}^\prime\otimes \ev(L_{(0)}\otimes a)\big)~,
\end{flalign}
for all $a\in A$, where in the third step we used the property $R(h\otimes g) = R(S(h)\otimes S(g))$,  for all $h,g\in H$,
see e.g.\ \cite[Lemma 2.2.2]{Majidbook}.
\end{rem}

\section*{Acknowledgements}
We thank Marco Benini, Giovanni Landi and Ryszard Nest
for helpful comments on the material presented in this paper. This work was completed while R.J.S.\ was
visiting the Centro de Matem\'atica, Computa\c{c}\~{a}o e
Cogni\c{c}\~{a}o of the Universidade de Federal do ABC in S\~ao Paulo,
Brazil during June--July 2016, whom he warmly thanks for support and hospitality during his stay there.
This work was supported by the Science and Technology Facilities Council [grant number ST/L000334/1],
the work of R.J.S.\ is supported in part by a Consolidated Grant from the UK STFC.
This work was also supported in part by the Action MP1405 QSPACE from 
the European Cooperation in Science and Technology (COST). 
G.E.B.\ is a Commonwealth Scholar, funded by the UK government. 
The work of A.S.\ was supported by a Research Fellowship of the Deutsche Forschungsgemeinschaft (DFG, Germany). 
The work of R.J.S.\ was supported in part by the Visiting Researcher Program
Grant 2016/04341-5 from the Funda\c{c}\~{a}o de Amparo \'a Pesquisa do
Estado de S\~ao Paulo (FAPESP, Brazil).

\appendix

\section{\label{app:technical}Technical details for Section
  \ref{sec:comparison}}

\subsection{Decomposable objects in ${}^H\MMM$}
Given any object $V$ in ${}^H\MMM$, we define
\begin{flalign}
V^{\bol{m}} := \big\{ v\in V \,:\, \rho^V(v) = t_{\bol{m}}\otimes v\big\}~,
\end{flalign}
for all $\bol{m}\in\bbZ^n$. Notice that $V^{\bol{0}}$ is the vector space of coinvariants
and that $V^{\bol{m}} \subseteq V$ are ${}^H\MMM$-subobjects, for all $\bol{m}$.
\begin{defi}\label{def:decomposable}
 An object $V$ in ${}^H\MMM$ is decomposable if the canonical ${}^H\MMM$-morphism
 \begin{flalign}
 \coprod_{\bol{m}\in\bbZ^n}\, V^{\bol{m}} \longrightarrow V~,~~\coprod_{\bol{m}}\, v_{\bol{m}} \longmapsto \sum_{\bol{m}}\, v_{\bol{m}}
 \end{flalign}
 is an isomorphism. We denote by ${}^H\MMM_{\mathrm{dec}}$ the full subcategory
of decomposables.
\end{defi}
\begin{lem}[Properties of decomposables]\label{lem:decomposableproperties}~
\begin{itemize}
\item[a)] Tensor products of decomposables are decomposable, i.e.\ 
${}^H\MMM_{\mathrm{dec}}$ is a  monoidal subcategory of ${}^H\MMM$.
\item[b)] Coproducts of decomposables are decomposable.
\item[c)] ${}^H\MMM$-subobjects of decomposables are decomposable.
\end{itemize}
\end{lem}
\begin{proof}
To prove item a), note that for $V,W$ decomposable we have
\begin{flalign}
V\otimes W \simeq \coprod_{\bol{m}\in \bbZ^{n}} \, \Big(\,
\coprod_{\bol{n}\in\bbZ^n} \, \big(V^{\bol{n}}\otimes
W^{\bol{m}-\bol{n}}\big)\, \Big)~,
\end{flalign}
hence $V\otimes W$ is decomposable with
\begin{flalign}
(V\otimes W)^{\bol{m}} = \coprod_{\bol{n}\in\bbZ^n}\, \big(V^{\bol{n}}\otimes W^{\bol{m}-\bol{n}}\big)~.
\end{flalign}
The monoidal unit object $\bbK_{\bol{0}}$ is clearly decomposable.
Items b) and c) are obvious.
\end{proof}
\begin{lem}\label{lem:Aisdecompo}
Let $A$ be an object in ${}^H\AAA_{\mathrm{fp}}$. Then the left $H$-comodule underlying $A$ is 
decomposable.
\end{lem}
\begin{proof}
Let us start with the case where $A = F_{\bol{m_1},\dots,\bol{m}_N}$ is a free ${}^H\AAA$-algebra.
As $F_{\bol{m_1},\dots,\bol{m}_N} \simeq F_{\bol{m}_1} \sqcup \cdots\sqcup F_{\bol{m}_M} $,
where $\sqcup$ denotes the coproduct in ${}^H\AAA_{\mathrm{fp}}$ (given explicitly by $\otimes$), 
we can use Lemma \ref{lem:decomposableproperties} a) and reduce the problem to showing that $F_{\bol{m}}$ is decomposable.
Notice that
\begin{flalign}
{F_{\bol{m}}}^{\bol{n}} = \begin{cases}
\mathrm{span}_{\bbK}(x^k)\simeq \bbK_{\bol{n}} & ~,~~\text{for }\bol{n} = k\,\bol{m}~,~k\in \bbZ_{\geq 0}~,\\
0 &~,~~\text{otherwise}~,
\end{cases}
\end{flalign}
where $x$ denotes the generator of $F_{\bol{m}}$ with $H$-coaction $x\mapsto t_{\bol{m}}\otimes x$.
The canonical ${}^H\MMM$-morphism reads as
\begin{flalign}
\coprod_{\bol{n}\in\bbZ^n}\, {F_{\bol{m}}}^{\bol{n}} \simeq \coprod_{k\in\bbZ_{\geq 0}}\, \bbK_{k\,\bol{m}}\longrightarrow F_{\bol{m}}~,~~\coprod_{k}\, c_k \longmapsto \sum_{k}\, c_k\,x^k~,
\end{flalign}
and it is easy to see that it is an isomorphism. 
\sk

For the case where $A = F_{\bol{m}_1,\dots,\bol{m}_N}/I$ is finitely
presented, we use the property that
$F_{\bol{m}_1,\dots,\bol{m}_N}$ is decomposable and hence so is the ${}^H\AAA$-ideal 
$I\subseteq F_{\bol{m}_1,\dots,\bol{m}_N}$. 
Consequently, the quotient $A = F_{\bol{m}_1,\dots,\bol{m}_N}/I$ is decomposable as well.
\end{proof}
\begin{cor}\label{cor:derisdecomposable}
Let $A$ be an object in ${}^H\AAA_{\mathrm{fp}}$. Then $\der(A)$ is decomposable.
\end{cor}
\begin{proof}
Recalling the definition of $\der(A)$ in \eqref{eqn:derf.g.}, 
the claim follows from the fact that $A$ is decomposable (cf.\ Lemma \ref{lem:Aisdecompo}),
and Lemma \ref{lem:decomposableproperties} b) and c).
\end{proof}

\subsection{Embedding of ${}^H\MMM$ into ${}^H\GGG$}
We first define a functor
\begin{subequations}\label{eqn:jfunctor}
\begin{flalign}
j : {}^H\MMM \longrightarrow \mathrm{PSh}({}^H\SSS)~.
\end{flalign}
To an object $V$ in ${}^H\MMM$ the functor $j$ assigns the presheaf
$j(V) : {}^H\SSS^\op \to \Set$ that acts on objects $X_B$ as
\begin{flalign}
j(V)(X_B) := (B\otimes V)^{\bol{0}}
\end{flalign}
and on morphisms $f: X_B\to X_C$ as
\begin{flalign}
j(V)(f) := (f^\ast \otimes \id_V) \,:\,  (C\otimes V)^{\bol{0}} \longrightarrow  (B\otimes V)^{\bol{0}}~.
\end{flalign}
\end{subequations}
To a morphism $L : V\to W$ in ${}^H\MMM$ the functor $j$ assigns the presheaf
morphism $j(L) : j(V) \to j(W)$  given by the natural transformation with components
\begin{flalign}\label{eqn:jLXB}
j(L)_{X_B} := (\id_{B}\otimes L)\,:\, (B\otimes V)^{\bol{0}} \longrightarrow (B\otimes W)^{\bol{0}}~.
\end{flalign}
\begin{propo}\label{propo:j(V)sheaf}
For any object $V$ in ${}^H\MMM$ the presheaf $j(V)$ is a sheaf. Hence \eqref{eqn:jfunctor}
induces a functor $j : {}^H\MMM \to {}^H\GGG$.
\end{propo}
\begin{proof}
Given any ${}^H\SSS$-Zariski covering family $\{f_i : X_{B[s_i^{-1}]} \to X_B^{}\}$,
we have to verify the sheaf condition \eqref{eqn:sheafcondition},
i.e.\ that the diagram
\begin{flalign}
\xymatrix{
(B\otimes V)^{\bol{0}} ~\ar[r] & ~\prod\limits_{i}\,  (B[s_i^{-1}]\otimes V)^{\bol{0}}  ~\ar@<-0.5ex>[r]\ar@<0.5ex>[r]& ~\prod\limits_{i,j}\, (B[s_i^{-1},s_j^{-1}]\otimes V)^{\bol{0}}~
}
\end{flalign}
is an equalizer in $\Set$. This follows from the same argument that we
have used in the second paragraph of the proof of 
Proposition~\ref{propo:fullyfaithfulYoneda}.
\end{proof}

For any object $X_B$ in ${}^H\SSS$, the set $j(V)(X_B) = (B\otimes V)^{\bol{0}}$
is a $B^{\bol{0}}$-module with Abelian group structure induced by the vector space structure of
$B\otimes V$ and $B^{\bol{0}}$-action
given by
\begin{flalign}\label{eqn:B0ModjV}
B^{\bol{0}} \times (B\otimes V)^{\bol{0}}\longrightarrow (B\otimes V)^{\bol{0}}~,~~(b,b^\prime \otimes v)\longmapsto
b\cdot (b^\prime\otimes v) := (b\,b^\prime\, )\otimes v~.
\end{flalign}
These structures are natural with respect to ${}^H\SSS$-morphisms $f : X_{B}\to X_{C}$, i.e.\
\begin{flalign}
(f^\ast\otimes\id_V) \big(c\cdot (c^\prime\otimes v)\big) = f^\ast(c)\cdot\big((f^\ast\otimes\id_V) (c^\prime\otimes v)\big)~,
\end{flalign}
for all $c\in C^{\bol{0}}$, $c^\prime\in C$ and $v\in V$,
hence they endow $j(V)$ with the structure of a $\underline{K}$-module object in ${}^H\GGG$. For any 
${}^H\MMM$-morphism $L:V\to W$ the ${}^H\GGG$-morphism $j(L) : j(V) \to j(W)$ 
is compatible with this $\underline{K}$-module object structure, i.e.\
\eqref{eqn:jLXB} is a $B^{\bol{0}}$-module morphism, for all objects $X_B$ in ${}^H\SSS$.
We have thereby obtained
\begin{propo}
With respect to the $\underline{K}$-module object structures on $j(V)$ introduced above,
$j: {}^H\MMM\to \Mod_{\underline{K}}({}^H\GGG)$  is a functor with values in
the category $\Mod_{\underline{K}}({}^H\GGG)$ of $\underline{K}$-module objects in ${}^H\GGG$.
\end{propo}
\begin{rem}\label{rem:notmonoidal}
The functor $j$ is not a monoidal functor, i.e.\ the object 
$j(V\otimes W)$ is in general not isomorphic to $j(V)\otimes j(W)$, where 
the tensor product in $\Mod_{\underline{K}}({}^H\GGG)$ is given by
\begin{flalign}
(j(V)\otimes j(W))(X_B) := j(V)(X_B)\otimes_{B^\bol{0}} j(W)(X_B)~,
\end{flalign}
for all objects $X_B$ in ${}^H\SSS$. For example, take $B=\bbK$
then $j(V\otimes W)(X_{\bbK}) = (V\otimes W)^{\bol{0}}$ but
$(j(V)\otimes j(W))(X_{\bbK}) = V^\bol{0}\otimes W^{\bol{0}}$.
However, there exists a $\Mod_{\underline{K}}({}^H\GGG)$-morphism
\begin{subequations}\label{eqn:psimap}
\begin{flalign}
\psi : j(V)\otimes j(W) \longrightarrow j(V\otimes W)~,
\end{flalign}
for all objects $V,W$ in ${}^H\MMM$. The components of $\psi$ are given by
\begin{flalign}
\nn \psi_{X_B} : (B\otimes V)^{\bol{0}}\otimes_{B^{\bol{0}}} (B\otimes W)^{\bol{0}}&\longrightarrow (B\otimes V\otimes W)^{\bol{0}}~,\\
(b\otimes v)\otimes_{B^{\bol{0}}} (b^\prime\otimes w)&\longmapsto R(b^\prime_{(-1)}\otimes v_{(-1)}) ~ (b\,b^\prime_{(0)})\otimes v_{(0)}\otimes w~,
\end{flalign}
\end{subequations}
for all objects $X_B$ in ${}^H\SSS$.
\end{rem}

Let now $V$ be decomposable, i.e.\ an object in ${}^H\MMM_{\mathrm{dec}}$. 
Because any object $B$ in ${}^H\AAA_{\mathrm{fp}}$
is decomposable as well (cf.\ Lemma \ref{lem:Aisdecompo}), we obtain
\begin{flalign}\label{eqn:jVXBexplicit}
j(V)(X_B) \simeq \coprod_{\bol{n}\in\bbZ^n} \, \big(B^{\bol{n}}\otimes V^{-\bol{n}}\big)~.
\end{flalign}
For the special case where $B = F_{\bol{m}}$ is the free ${}^H\AAA$-algebra with one generator
with coaction $x\mapsto t_{\bol{m}}\otimes x$, we use 
$B\simeq \coprod_{k\in \bbZ_{\geq 0}} \, \bbK_{k\,\bol{m}}$ to simplify this expression further to
\begin{flalign}\label{eqn:jVXFmexplicit}
j(V)(X_{F_{\bol{m}}}) \simeq  \coprod_{k\in\bbZ_{\geq 0}} \, V^{-k\,\bol{m}}~,
\end{flalign}
where the coproducts here are in the category of vector spaces.
Using this explicit characterization, we can establish the main result
of this appendix.
\begin{theo}
For any two objects $V,W$ in ${}^H\MMM_{\mathrm{dec}}$ there is a bijection of $\Hom$-sets
\begin{flalign}
\Hom_{{}^H\MMM}(V,W) \simeq \Hom_{\Mod_{\underline{K}}({}^H\GGG)} (j(V),j(W))~. 
\end{flalign}
Thus the restricted functor $j : {}^H\MMM_{\mathrm{dec}} \to \Mod_{\underline{K}}({}^H\GGG)$
to the full subcategory of decomposables ${}^H\MMM_{\mathrm{dec}}$ is fully faithful.
\end{theo}
\begin{proof}
Let $\eta : j(V)\to j(W)$ be any morphism in $\Mod_{\underline{K}}({}^H\GGG)$.
The components
\begin{flalign}
\eta_{X_B} : (B\otimes V)^{\bol{0}}\longrightarrow (B\otimes W)^{\bol{0}}~
\end{flalign}
are $B^{\bol{0}}$-module morphisms, for all objects $X_B$ in ${}^H\SSS$,
such that for any ${}^H\SSS$-morphism $f : X_B\to X_C$ the diagram
\begin{flalign}\label{eqn:naturaldiagramtmp}
\xymatrix{
\ar[d]_-{f^\ast\otimes\id_V}(C\otimes V)^{\bol{0}} \ar[rr]^-{\eta_{X_C}} && (C\otimes W)^{\bol{0}}\ar[d]^-{f^\ast\otimes\id_W}\\
(B\otimes V)^{\bol{0}} \ar[rr]_-{\eta_{X_B}} && (B\otimes W)^{\bol{0}}
}
\end{flalign}
commutes.
\sk

We first show that $\eta$ is uniquely determined by the components $\eta_{X_{F_{\bol{m}}}}$,
for all free ${}^H\AAA$-algebras $F_{\bol{m}}$ with one generator.
Using \eqref{eqn:jVXBexplicit}, we find that $\eta_{X_B}$ is specified by its action
on elements of the form $b\otimes v \in B^{\bol{n}}\otimes V^{-\bol{n}}$, for all $\bol{n}$.
Given any such element, we define an ${}^H\AAA_{\mathrm{fp}}$-morphism
$f^\ast : F_{\bol{n}} \to B$ by sending $x\mapsto b$. 
(Notice that the morphism $f^\ast$ depends on the chosen element $b\otimes v$.)
Then the commutative diagram \eqref{eqn:naturaldiagramtmp}
implies that $\eta_{X_B}(b\otimes v) = (f^\ast\otimes \id_W)(\eta_{X_{F_{\bol{n}}}}(x\otimes v))$,
hence the value of $\eta_{X_B}$ at $b\otimes v$ is fixed by $\eta_{X_{F_{\bol{n}}}}$.
As $b\otimes v$ was arbitrary, we find that $\eta$ is uniquely determined
by the components $\{\eta_{X_{F_{\bol{m}}}} : \bol{m}\in\bbZ^n\}$.
\sk

In the next step we show that the components $\{\eta_{X_{F_{\bol{m}}}} : \bol{m}\in\bbZ^n\}$
are uniquely determined by an ${}^H\MMM$-morphism $L : V\to W$.
Consider the ${}^H\AAA_{\mathrm{fp}}$-morphism $f^\ast : F_{\bol{m}}\to F_{\bol{m}}$
defined by $x\mapsto c\,x$, where $c\in\bbK$ is an arbitrary constant.
Using \eqref{eqn:jVXFmexplicit} and the commutative diagram \eqref{eqn:naturaldiagramtmp} corresponding
to this morphism, we obtain a commutative diagram
\begin{flalign}
\xymatrix{
\ar[d] \coprod\limits_{k\in\bbZ_{\geq 0}} \, V^{-k\,\bol{m}} \ar[rr]^-{\eta_{X_{F_{\bol{m}}}}} && \coprod\limits_{k\in\bbZ_{\geq 0}}\, W^{-k\,\bol{m}}\ar[d]\\
\coprod\limits_{k\in\bbZ_{\geq 0}} \, V^{-k\,\bol{m}}
\ar[rr]_-{\eta_{X_{F_{\bol{m}}}}} && \coprod\limits_{k\in\bbZ_{\geq
    0}} \, W^{-k\,\bol{m}}
}
\end{flalign}
The vertical arrows map elements $v\in  V^{-k\,\bol{m}}$ 
to $c^k\,v \in \coprod_{k\in\bbZ_{\geq 0}}\, V^{-k\,\bol{m}}$ (and similarly for $w\in W^{-k\,\bol{m}}$), 
where the power in $c^k$ depends on the term in the coproduct.
Hence by ${F_{\bol{m}}}^{\bol{0}}$-linearity of $\eta_{X_{F_{\bol{m}}}}$ (which in particular 
implies $\bbK$-linearity), 
we find that $\eta_{X_{F_{\bol{m}}}}$ decomposes into $\bbK$-linear maps
\begin{flalign}
L_{\bol{m}, k} :  V^{-k\,\bol{m}}  \longrightarrow W^{-k\,\bol{m}} ~.
\end{flalign}
It remains to show that $L_{\bol{m}, k} = L_{k\,\bol{m},1}$, for all $\bol{m}\in \bbZ^n$ and all $k\in\bbZ_{\geq 0}$.
Consider the ${}^H\AAA_{\mathrm{fp}}$-morphism $f^\ast : F_{k\,\bol{m}}\to F_{\bol{m}}$
defined by $x\mapsto x^k$. The corresponding commutative diagram \eqref{eqn:naturaldiagramtmp}
then relates $\eta_{X_{F_{k\,\bol{m}}}} $ to $\eta_{X_{F_{\bol{m}}}} $ and we obtain  the desired result
$L_{\bol{m}, k} = L_{k\,\bol{m},1}$. This defines a unique ${}^H\MMM$-morphism
\begin{flalign}
L:= \coprod_{\bol{m}\in\bbZ^n}\, L_{\bol{m},1} :
\coprod_{\bol{m}\in\bbZ^n}\, V^{-\bol{m}} \longrightarrow \coprod_{\bol{m}\in\bbZ^n}\, W^{-\bol{m}}~,
\end{flalign}
and hence by the assumption that $V$ and $W$ are decomposable 
also a unique ${}^H\MMM$-morphism $L : V\to W$. 
\end{proof}

\end{document}